\documentclass[11pt,b5paper,notitlepage]{article}
\usepackage[b5paper, margin={0.5in,0.65in}]{geometry}

\usepackage{amsmath,amscd,amssymb,amsthm,mathrsfs,amsfonts,layout,indentfirst,graphicx,caption,mathabx, stmaryrd,appendix,calc,imakeidx,upgreek} 
\usepackage[dvipsnames]{xcolor}
\usepackage{palatino}  
\usepackage{slashed} 
\usepackage{mathrsfs} 
\usepackage{extarrows} 
\usepackage{enumitem} 

\usepackage{fancyhdr} 

\usepackage{ulem}  

\usepackage{tikz-cd}
\usepackage[nottoc]{tocbibind}   

\makeindex

\usepackage{lipsum}
\let\OLDthebibliography\thebibliography
\renewcommand\thebibliography[1]{
	\OLDthebibliography{#1}
	\setlength{\parskip}{0pt}
	\setlength{\itemsep}{2pt} 
}

\allowdisplaybreaks  
\usepackage{latexsym}
\usepackage{chngcntr}
\usepackage[colorlinks,linkcolor=blue,anchorcolor=blue, linktocpage,
]{hyperref}
\hypersetup{ urlcolor=cyan,
	citecolor=[rgb]{0,0.5,0}}

\counterwithin{figure}{subsection}

\theoremstyle{definition}
\newtheorem{df}{Definition}[section]

\newtheorem{rem}[df]{Remark}

\theoremstyle{plain}
\newtheorem{thm}[df]{Theorem}

\newtheorem{pp}[df]{Proposition}

\newtheorem{lm}[df]{Lemma}

\newtheorem{Mthm}{Main Theorem}

\newcommand{\fk}{\mathfrak}
\newcommand{\mc}{\mathcal}
\newcommand{\wtd}{\widetilde}
\newcommand{\wht}{\widehat}

\newcommand{\ovl}{\overline}

\newcommand{\End}{\mathrm{End}} 
\newcommand{\id}{\mathbf{1}}
\newcommand{\Hom}{\mathrm{Hom}}

\newcommand{\ev}{\mathrm{ev}}
\newcommand{\coev}{\mathrm{coev}}
\newcommand{\opp}{\mathrm{opp}}
\newcommand{\Rep}{\mathrm{Rep}}

\newcommand{\Diffp}{\mathrm{Diff}^+}
\newcommand{\Diff}{\mathrm{Diff}}
\newcommand{\PSU}{\mathrm{PSU}(1,1)}
\newcommand{\Vir}{\mathrm{Vir}}

\newcommand{\Span}{\mathrm{Span}}

\newcommand{\bk}[1]{\langle {#1}\rangle}
\newcommand{\GA}{\mathscr G_{\mathcal A}}

\newcommand{\scr}{\mathscr}

\newcommand{\Jtd}{\widetilde{\mathcal J}}

\newcommand{\im}{\mathbf{i}}

\newcommand{\RepA}{\mathrm{Rep}(\mathcal A)}

\newcommand{\RepfA}{\mathrm{Rep}^{\mathrm f}(\mathcal A)}

\newcommand{\RepL}{\mathrm{Rep}^{\mathrm{L}}}
\newcommand{\HomL}{\mathrm{Hom}^{\mathrm{L}}}
\newcommand{\EndL}{\mathrm{End}^{\mathrm{L}}}

\newcommand{\mbb}{\mathbb}

\newcommand{\blt}{\bullet}

\newcommand{\Cbb}{\mathbb C}

\newcommand{\Zbb}{\mathbb Z}

\newcommand{\Rbb}{\mathbb R}

\newcommand{\UPSU}{\widetilde{\mathrm{PSU}}(1,1)}
\newcommand{\Sbb}{{\mathbb S}}
\newcommand{\Gc}{\mathscr G_c}
\newcommand{\Obj}{\mathrm{Obj}}

\usepackage{tipa} 

\usepackage{tipx}

\numberwithin{equation}{section}

\title{On a Connes Fusion Approach to Finite Index Extensions of Conformal Nets}
\author{{\sc Bin Gui}
}
\date{}
\begin{document}\sloppy 
	\pagenumbering{arabic}
	\setcounter{section}{-1}

	\maketitle

\newcommand\blfootnote[1]{%
	\begingroup
	\renewcommand\thefootnote{}\footnote{#1}%
	\addtocounter{footnote}{-1}%
	\endgroup
}


\begin{abstract}
In the framework of Connes fusion, we discuss the relationship between the  (non-necessarily local or irreducible) finite index extensions $\mc B$ of an irreducible local M\"obius covariant net $\mc A$ on $\Sbb^1$ and the $C^*$-Frobenius algebras $Q$ in $\RepA$ (the $C^*$-tensor category  of  $\mc A$-modules). We explain how to prove the 1-1 correspondence between $Q$ and $\mc B$ in this framework, and show that the $C^*$-category $\RepL(Q)$ of left $Q$-modules is isomorphic to the one $\Rep(\mc B,\mc A)$ of ``$(\mc B,\mc A)$-modules". When $\mc B$ is an irreducible local M\"obius extension, this reduces to a braided $C^*$-tensor isomorphism of the categories of dyslectic $Q$-modules and $\mc B$-modules. We also establish a 1-1 correspondence between the (non-necessarily unitary)  isomorphisms of finite index extensions of $\mc A$ and the isomorphisms of $C^*$-Frobenius algebras as defined in \cite{NY18}.
\end{abstract}

\tableofcontents





\section{Introduction}

Q-systems ($\approx$ $C^*$-Frobenius algebras) were introduced by Longo \cite{Lon94} and are powerful tools for the study of local and non-local finite index extensions of conformal nets, full and boundary conformal field theory, and defects \cite{LR95,KL04,KLPR07,LR04,LR09,CKL13,BKL15,BKLR15,BKLR16,BDH19}. Finite index extensions of conformal nets were investigated mainly  in the framework of DHR (Doplicher-Haag-Roberts \cite{DHR71,DHR74}) superselection theory \cite{FRS89,FRS92},  but not often in the Connes fusion setting \cite{Con80,Sau83,Was98,BDH17}. Though the Connes fusion approach  to extensions of factors is well established (cf. for instance \cite{Mas97,Yuan19,GY20,CHPJP21}), its generalization to extensions of \textit{nets of factors} (for instance, extensions of conformal nets) is not straightforward.   

It is a main goal of this article to give detailed proofs  of some basic results on the relationship between $C^*$-Frobenius algebras and finite index extensions of conformal nets in the Connes fusion framework.   Moreover, the results we shall prove are slightly more general than those already exist in the literature. Specifically:
\begin{itemize}
\item We do not assume the conformal nets are completely rational.
\item Our study is not limited to dualizable representations.
\item The extensions are not assumed to be irreducible.
\item M\"obius covariance is assumed only for the smaller net $\mc A$, but not initially for the (finite index) extension $\mc B$; rather, it will be shown to hold as a consequence (see Thm. \ref{lb5} and \ref{lb11}).
\end{itemize}
Indeed, the only finiteness condition we assume is that of the extension $\mc A\subset\mc B$.

Let $\mc A$ be any irreducible local M\"obius covariant net (irreducible M\"obius net for short), and let $\RepA$ be the braided $C^*$-tensor category of (normal) $\mc A$-modules defined by Connes fusion. For each $C^*$-Frobenius algebra $Q$ in $\RepA$, we have constructed a finite index extension $\mc B_Q$ in \cite{Gui21b} (see also Thm. \ref{lb5}). Here, we show that any finite index extension arises in this way	(Thm. \ref{lb11}).  In the DHR superselection setting, these results are due to  \cite[Thm. 4.9]{LR95} (see also \cite[Sec. 5.2]{BKLR15}).

We show that the $C^*$-tensor category $\RepL(Q)$ of left $Q$-modules is isomorphic to the one $\Rep(\mc B_Q,\mc A)$ of (naturally defined) $(\mc B_Q,\mc A)$-modules (Main Thm. \ref{lb12}). When $\mc B_Q$ is a local M\"obius extension (equivalently, when $Q$ is a M\"obius covariant $\mc A$-module and is commutative with trivial twist operator, cf. Thm. \ref{lb25}), this isomorphism reduces to an isomorphism of braided $C^*$-tensor categories $\Rep^0(Q)\simeq\Rep(\mc B_Q)$ where $\Rep^0(Q)$ is the category of dyslectic $Q$-modules, and $\Rep(\mc B_Q)$ is the category of usual $\mc B_Q$-modules. This is Main Thm. \ref{lb24}. In the DHR setting, this result was proved when $\mc A$ is completely rational and the categories are assumed to contain only dualizable $\mc A$-modules (cf. \cite{Mug10} and \cite[Prop. 6.4]{BKL15}). Our result does not assume these conditions.

Our third main result is not as familiar as the previous ones. We introduce a natural definition of isomorphisms of extensions of $\mc A$. These isomorphisms are in general not implemented by unitary operators (unless when the extensions are irreducible).  We construct a 1-1 correspondence between (1) the isomorphisms $\varphi:\mc B^a\rightarrow\mc B^b$ (with respect to $\mc A$)  of finite index extensions of $\mc A$ and (2) the isomorphisms $V:Q^a\rightarrow Q^b$ of $C^*$-Frobenius algebras in $\RepA$ in the sense of \cite{NY18} (cf. Thm. \ref{lb13} or Main Thm. \ref{lb26}). Such  $V$ is not only an isomorphism of algebra objects in the usual sense; it also satisfies that $V^*V$ is a homomorphism of left $Q^a$-modules. We call such $V$ a \textit{left isomorphism} in this article. This result is closely related to choosing two different faithful normal conditional expectations for a finite index subfactor $\mc N\subset\mc M$. See the long Remark \ref{lb15}.

Many important topics are not discussed in this article. For instance, when $\mc B_Q$ is local, we do not show that the $C^*$-category equivalence $\RepL(Q)\simeq\Rep(\mc B_Q,\mc A)$ is an equivalence of $C^*$-tensor categories (though we have proved this for $\Rep^0(Q)\simeq \Rep(\mc B_Q)$), since the theory of categorical extensions for $(\mc B_Q,\mc A)$-modules has not been established yet. In particular, we are not yet ready to translate the tensor-categorical results on orbifold CFT in \cite{Mug05} to the Connes fusion framework. (Such translation would be useful for the orbifold VOA-conformal net correspondence.) Also, to keep the article relatively short, we do not discuss either full or boundary CFT, or  fermionic conformal nets here. But many crucial  techniques for exploring these topics through Connes fusion are already given in the article.

\subsection*{Acknowledgment}

I would like to thank Andr\'e Henriques and David Penneys for helpful discussions, and especially Andr\'e Henriques for encouraging me to write this note.

\section{Backgrounds}\label{lb18}

All Hilbert spaces are assumed to be separable.	For a Hilbert space $\mc H$, $\bk{\cdot|\cdot}$ denotes its inner product whose first variable is linear and the second one antilinear. $\End(\mc H_i,\mc H_2)$ denotes the set of bounded linear operators from a Hilbert space $\mc H_i$ to another one $\mc H_j$. We follow the usual convention that if $M$ is a von Neumann algebra on a Hilbert space then $M'$ denotes its commutant.

If $\mc P,\mc Q,\mc R,\mc S$ are Hilbert spaces, $A,B,C,D$ are bounded linear operators whose domains and codomains are indicated by the following diagram,
\begin{equation}
	\begin{tikzcd}
		\mc P \arrow[r, "C"] \arrow[d, "A"'] & \mc Q \arrow[d, "B"]\\
		\mc R\arrow[r, "D"] &\mc S
	\end{tikzcd}
\end{equation}
we say that this diagram \textbf{commutes adjointly} if $DA=BC$ and $CA^*=B^*D$. If either $A,B$ or $C,D$ are unitary, then commuting implies adjoint commuting.

\subsection*{M\"obius nets and conformal nets}

Let $\mc J$ be the set of all non-empty non-dense open intervals in the unit circle $\mbb S^1$. If $I\in\mc J$, then $I'\in\mc J$ is by definition the interior of the complement of $I$. The group $\Diffp(\mbb S^1)$ of orientation-preserving diffeomorphisms of $\mbb S^1$ contains the subgroup $\PSU$ of M\"obius transforms of $\mbb S^1$. We let $\UPSU$ and $\scr G$ be respectively the universal coverings of $\PSU$ and $\Diffp(\mbb S^1)$. 

If $I\in\mc J$, we let $\Diff(I)$ be the subgroup of all $g\in\Diffp(\mbb S^1)$ acting as identity on $I'$. We let $\scr G(I)$ be the connected component of the inverse image of $\Diff(I)$ under $\scr G\rightarrow\Diffp(\Sbb^1)$ that contains the identity.

By a  \textbf{M\"obius net} $\mc A$, or more precisely, a local M\"obius covariant net, we mean that for each $I\in\mathcal J$ there is a von Neumann algebra $\mathcal A(I)$ acting on a fixed separable Hilbert space $\mathcal H_0$, such that the following conditions hold:
\begin{enumerate}[label=(\alph*)]
\item (Isotony) If $I_1\subset I_2\in\mathcal J$, then $\mathcal A(I_1)$ is a von Neumann subalgebra of $\mathcal A(I_2)$.
\item (Locality) If $I_1,I_2\in\mathcal J$ are disjoint, then $\mathcal A(I_1)$ and $\mathcal A(I_2)$ commute.
\item (M\"obius covariance) We have a strongly continuous  unitary representation $U$ of $\PSU$ on $\mc H_0$ such that for any $g\in\PSU, I\in\mc J,$, 
\begin{align*}
	U(g)\mc A(I)U(g)^*=\mc A(gI).
\end{align*}
\item (Positivity of energy) The generator $L_0$ of the rotation subgroup $\varrho$  is positive.
\item There exists a $\PSU$-invariant unit vector $\Omega\in\mc H_0$, called the \textbf{vacuum vector} which is  cyclic under the action of $\bigvee_{I\in\mathcal J}\mathcal A(I)$ (the von Neumann algebra generated by all $\mathcal A(I)$).
\end{enumerate}

$\mc A$ is \textbf{additive}, namely, if $\fk I$ is a set of intervals with union $J\in\mc J$, then $\mc A(J)$ is the von Neumann algebra generated by $\bigcup_{I\in\fk I}\mc A(I)$. $\mc A$ satisfies \textbf{Haag duality}, i.e., that $\mc A(I')=\mc A(I)'$ for each $I\in \mc J$. Moreover, we have the \textbf{Reeh-Schlieder property}, which says that $\Omega$ is  cyclic separating vector for each $\mc A(I)$. We say that $\mc A$ is \textbf{irreducible} if $\Cbb \id$ are the only operators on $\mc H_0$ commuting with the action of $\mc A(I)$ for all $I\in\mc J$. This is equivalent to that $\Cbb\Omega$ are the only $\PSU$-invariant vectors in $\mc H_0$, and also equivalent to that each $\mc A(I)$ is a type III$_1$-factor.  See \cite[Sec. 1]{GL96} for details.

A M\"obius net $\mc A$ is called a \textbf{conformal (covariant) net} if the representation $U$ of $\PSU$ on $\mc H_0$ can be extended to a strongly continuous projective unitary representation $U$ of $\Diffp(\mbb S^1)$ on $\mathcal H_0$, such that for any $g\in \Diffp(\mbb S^1),I\in\mathcal J$, and any representing element $V\in\mathcal U(\mathcal H_0)$ of $U(g)$,
\begin{align*}
	V\mathcal A(I)V^*=\mathcal A(gI).
\end{align*}
Moreover, if $g\in\Diff(I)$ then $V\in\mc A(I)$.

If $\mc A$ is a M\"obius net, a (normal untwisted) \textbf{$\mc A$-module} is a pair $(\pi_i,\mc H_i)$, or simply $\mc H_i$, such that $\mc H_i$ is a Hilbert space, and that for any $I\in\mc J$ there is an (automatically normal)  representation $\pi_{i,I}$ of $\mc A(I)$ on $\mc H_i$ such that $I\subset J\Rightarrow \pi_{i,J}|_{\mc A(I)}=\pi_{i,I}$. We use symbols $\mc H_i,\mc H_j,\mc H_k,\dots$ for $\mc A$-modules. For $x\in\mc A(I)$ and $\xi\in\mc H_i$, $\pi_{i,I}(x)\xi$ will be written as $x\xi$ when the context is clear. An $\mc A$-module is called \textbf{M\"obius covariant}, if there is a strongly continuous unitary representation $U_i$ of $\UPSU$ on $\mc H_i$ such that $U(g)\pi_{i,I}(x)U(g)^*=\pi_{i,gI}(U(g)xU(g)^*)$ for each $g\in\UPSU$, $I\in\mc J$, and $x\in\mc A(I)$. Clearly $\mc H_0$ is an $\mc A$-module, called the \textbf{vacuum module} of $\mc A$.

If $\mc A$ is a conformal net, any $\mc A$-module $\mc H_i$ is automatically M\"obius covariant. Indeed, a stronger property holds. Let $\mc U(\mc H_0)$ be the topological group of unitary operators on $\mc H$, equipped with the strong operator topology. Let $\GA$ be the subgroup of all $(g,V)\in\scr G\times\mc U(\mc H_0)$ such that $V$ represents $U(g)$. Then we have a central extension
\begin{align*}
	1\rightarrow \Sbb^1\rightarrow \GA\rightarrow\scr G\rightarrow 1.
\end{align*}
$\GA$ naturally acts on $\mc H_0$. As explained in \cite[Thm. 2.2]{Gui21a}, the results in \cite{Hen19} imply that   each $\mc A$-module $\mc H_i$ is \textbf{conformal covariant}, which means that there is a (necessarily unique) strongly continuous unitary representation $U_i$ of $\GA$ on $\mc H_i$ such that
\begin{align}
U_i(g)=\pi_{i,I}(U(g))	\label{eq1}
\end{align}
for each $I\in\mc J$ and each $g$ belonging to 
\begin{align*}
\GA(I):=\text{the inverse image of $\scr G(I)$ under $\GA\rightarrow\scr G$}.	
\end{align*}
It follows easily (cf. the proof of \cite[Cor. 2.6]{Gui21a}) that   $U_i(g)\pi_{i,I}(x)U_i(g)^*=\pi_{i,gI}(gx)$ for each $I\in\mc J,x\in\mc A(I),g\in\GA$.

Thus, if $\mc A$ is conformal and $\mc H_i$ is an $\mc A$-module, we fix the continuous representation of $\UPSU$ to be the one inherited from that of $\GA$.

If $\mc A$ is an irreducible conformal net, then $\mc A$ contains the subnet  $I\in\mc J\mapsto \{U(g):g\in\Diff(I)\}''$ acting on the subspace generated by $\Omega$. This (M\"obius) subnet is clearly irreducible.  In that case, it is the Virasoro net $\Vir_c$ of some central charge $c$ (cf. \cite[Thm. A.1]{Car04}). Let $\Gc:=\scr G_{\Vir_c}$ The map $(g,V)\in\GA\mapsto (g,V|_{\ovl{\Span(\Vir_c\Omega)}})\in\Gc$ is clearly a bijective contiuous group homomorphism. That its inverse is continuous follows from the fact that $\mc H_0$ is a conformal covariant $\Vir_c$-module. Thus, we shall make the identification of topological groups
\begin{align*}
	\Gc=\GA,\qquad \Gc(I)=\GA(I)	
\end{align*}
so that the central extension of $\scr G$ is only depending on the central charge $c$ of $\mc A$.

For any element $g$ of $\Diffp(\Sbb^1)$ or $\scr G$ or $\scr G_c$, we shall abbreviate $U_i(g)$ or $U(g)$ to $g$ when the context is clear.

For a M\"obius net $\mc A$, we let $\Rep(\mc A)$ be the $C^*$-category of $\mc A$-modules. The space of morphisms between two modules $\mc H_i,\mc H_j$ is
\begin{align*}
\Hom_{\mc A}(\mc H_i,\mc H_j)=\{T\in\Hom(\mc H_i,\mc H_j):T\pi_{i,I}(x)=\pi_{j,I}(x)T\text{ for each $I\in\mc J,x\in\mc A(I)$}\}.	
\end{align*}
Hence $\mc A$ is irreducible if and only if $\mc H_0$ is an irreducible $\mc A$-module.

\section*{Categorical extensions of $\mc A$}

In the remaining part of this article, \uwave{we always assume that $\mc A$ is an irreducible M\"obius net}.

There are two equivalent ways to make $\Rep(\mc A)$ a braided $C^*$-tensor category: the Doplicher-Haag-Roberts (DHR) superselection theory \cite{DHR71,DHR74,FRS89,FRS92}, and the Connes fusion \cite{Con80,Sau83,Was98,BDH17}. In this article, we focus on the latter. 

The braided $C^*$-tensor category $(\Rep(\mc A),\boxtimes,\ss)$ is uniquely determined by the existence of a categorical extension of $\mc A$ on $\Rep(\mc A)$. To recall the definition, we first introduce some terminology.

If $I\in\mc J$, an arg function $\arg_I:I\rightarrow\Rbb$ is a continuous function whose value at each $e^{\im\theta}\in I$ is in $\theta+2\pi\Zbb$. $(I,\arg_I)$ is called an \textbf{arg-valued interval}. Let
\begin{align*}
\Jtd=\text{the set of all arg-valued intervals}.	
\end{align*}
One may regard each $\wtd I\in\Jtd$ as an interval in the universal cover $\Rbb$ of $\Sbb^1$. Then the actions of $\PSU$ and $\Diffp(\Sbb^1)$ lift to actions of $\UPSU$ and $\scr G$ on $\Rbb$. Thus $\UPSU$ and $\scr G$ (and hence $\Gc$) act on $\Jtd$.  

We say that $\wtd J=(J,\arg_J)$ is \textbf{clockwise} to $\wtd I=(I,\arg_I)$ if   $\arg_I(z)-2\pi<\arg_J(\zeta)<\arg_I(z)$ for each $z\in I,\zeta\in J$. In particular, $I\cap J=\emptyset$. We mean
\begin{align*}
\wtd I	\subset\wtd J\quad\Leftrightarrow\quad \text{$I\subset J$ and $\arg_J|_I=\arg_I$}.
\end{align*}
In this case we say $\wtd J$ contains $\wtd I$. The \textbf{clockwise complement} of $\wtd I$ is defined to be
\begin{align*}
	\wtd I'=\text{the largest arg-valued interval clockwise to $\wtd I$}.
\end{align*} 
$\wtd I$ is called the \textbf{anticlockwise complement} of $\wtd I'$.

For each $\mc A$-module $\mc H_i$ and each $I\in\mc J$, we let
\begin{align*}
\mc H_i(I)=\Hom_{\mc A(I')}(\mc H_0,\mc H_i)\cdot\Omega
\end{align*}
where $\Hom_{\mc A(I')}(\mc H_0,\mc H_i)$ is the set of bounded operators from $\mc H_0$ to $\mc H_i$ intertwining the actions of $\mc A(I')$. Note that by Haag duality, $\mc H_0(I)=\mc A(I)\Omega$. In general,  $\mc H_i(I)$ is a dense subspace of $\mc H_i$. This is because every (normal) representation of $\mc A(I)$ (and in particular $\mc H_i$) is unitarily equivalent to a unitary subrepresentation of a direct sum of $\mc H_0$. (In fact, since $\mc A(I)$ is type III, every nonzero separable representation of $\mc A(I)$ is unitarily equivalent to $\mc H_i$.)

\begin{df}\label{lb23}
By a (closed and vector-labeled) \textbf{categorical extension} $\scr E=(\mc A,\Rep(\mc A),\boxtimes,\mc H)$ (where $\mc H$ denotes the association $I\mapsto \mc H_i(I)$ for each module $\mc H_i$), we mean the following: $\boxtimes$ is a $*$-bifunctor such that $(\RepA,\boxtimes)$ is a  $C^*$-tensor category.\footnote{This means in particular that the associators and the unitors are unitary isomorphisms, and for the morphisms we have $(F\boxtimes G)^*=F^*\boxtimes G^*$. If $\Rep(\mc A)$ is equipped with a braid structure such that the braiding isomorphisms are unitary, we say $\Rep(\mc A)$ is a braided $C^*$-tensor category.} (We suppress the associative isomorphisms and the unitors.)  Moreover,  $\scr E$ associates to any  $\mc H_i,\mc H_k\in\Obj(\Rep(\mc A))$ and any $\wtd I\in\Jtd,\fk \xi\in\mc H_i(I)$ bounded linear operators
\begin{gather*}
	L(\xi,\wtd I)\in\Hom_{\mc A(I')}(\mc H_k,\mc H_i\boxtimes\mc H_k),\\
	R(\xi,\wtd I)\in\Hom_{\mc A(I')}(\mc H_k,\mc H_k\boxtimes\mc H_i),
\end{gather*}
such that the following conditions are satisfied:
\begin{enumerate}[label=(\alph*)]
\item (Isotony) If $\wtd I_1\subset\wtd I_2\in\Jtd$, and $\xi\in\mc H_i(I_1)$, then $L(\xi,\wtd I_1)=L(\xi,\wtd I_2)$, $R(\xi,\wtd I_1)=R(\xi,\wtd I_2)$ when acting on any  $\mc H_k\in\Obj(\Rep(\mc A))$.
\item (Functoriality) If $\mc H_i,\mc H_j,\mc H_{j'}\in\Obj(\RepA)$, $\wtd I\in\Jtd$, $G\in\Hom_{\mc A}(\mc H_j,\mc H_{j'})$,  $\xi\in\mc H_i(I)$, and $\eta\in\mc H_j$, then
\begin{align}
	(\id_i\boxtimes G)L(\xi,\wtd I)\eta=L(\xi,\wtd I)G\eta,\qquad (G\boxtimes \id_i)R(\xi,\wtd I)\eta=R(\xi,\wtd I)G\eta.
\end{align}
\item (State-field correspondence) For any $\mc H_i\in\Obj(\RepA)$, under the identifications $\mc H_i=\mc H_i\boxtimes\mc H_0=\mc H_0\boxtimes\mc H_i$ defined by the unitors, the relation
\begin{align}
	L(\xi,\wtd I)\Omega=R(\xi,\wtd I)\Omega=\xi\label{eq14}
\end{align}
holds for any $\wtd I\in\Jtd,\xi\in\mc H_i(I)$. It follows immediately that when acting on $\mc H_0$, $L(\xi,\wtd I)$ equals $R(\xi,\wtd I)$ and is independent of $\arg_I$.
\item (Density of fusion products) If $\mc H_i,\mc H_k\in\Obj(\RepA),\wtd I\in\Jtd$, then the set $L(\mc H_i(I),\wtd I)\mc H_k$ spans a dense subspace of $\mc H_i\boxtimes\mc H_k$, and $R(\mc H_i(I),\wtd I)\mc H_k$ spans a dense subspace of $\mc H_k\boxtimes\mc H_i$. (Indeed, they span the full space $\mc H_i\boxtimes\mc H_k$ and $\mc H_k\boxtimes\mc H_i$ respectively.)
\item (Locality) For any $\mc H_k\in\Obj(\RepA)$, disjoint $\wtd I,\wtd J\in\Jtd$ with $\wtd I$ anticlockwise to $\wtd J$, and any $\xi\in\mc H_i(I),\eta\in\mc H_j(J)$, the following diagram commutes adjointly.
\begin{equation}
\begin{tikzcd}
\quad \mc H_k\quad \arrow[rr,"{R(\eta,\wtd J)}"] \arrow[d, "{L(\xi,\wtd I)}"'] &&\quad \mc H_k\boxtimes\mc H_j\quad \arrow[d, "{L(\xi,\wtd I)}"]\\
\mc H_i\boxtimes\mc H_k\arrow[rr,"{R(\eta,\wtd J)}"] &&\mc H_i\boxtimes\mc H_k\boxtimes\mc H_j
\end{tikzcd}
\end{equation}
\item (Braiding) There is a unitary linear map $\ss_{i,j}:\mc H_i\boxtimes\mc H_j\rightarrow\mc H_j\boxtimes \mc H_i$  for any $\mc H_i,\mc H_j\in\Obj(\RepA)$, such that  
\begin{align}
	\ss_{i,j} L(\xi,\wtd I)\eta=R(\xi,\wtd I)\eta\label{eq24}
\end{align}
whenever $\wtd I\in\Jtd,\xi\in\mc H_i(I)$, $\eta\in\mc H_j$.
\end{enumerate}
\end{df}

The above $\ss_{ij}$ is necessarily an $\mc A$-module isomorphism making $(\Rep(\mc A),\boxtimes,\ss)$ a braided $C^*$-tensor category. Moreover, such categorical extensions over the $C^*$-category $\RepA$ exist and are unique.  In particular, if we have another categorical extension $(\mc A,\RepA,\boxdot,\mc H)$ which determines a braiding $\sigma$, then $(\RepA,\boxtimes,\ss)\simeq(\RepA,\boxdot,\sigma)$. See Sec. 3.1-3.4, especially Thm. 3.4 and 3.10 of \cite{Gui21a}.\footnote{Although \cite{Gui21a} only discusses irreducible conformal nets, any result irrelavent to conformal covariance also holds for irreducible M\"obius nets. 
} In other words, the existence of the $L$ and $R$ operations satisfying the above axioms characterizes the braided $C^*$-tensor structure.

$\scr E$ was constructed in \cite{Gui21a} using Connes fusion. For a brief explanation of the construction, see \cite[Sec. A]{Gui21b}.

We give some useful facts that will be used later in this article. First, if the irreducible M\"obius net $\mc A$ is conformal, then $\scr E$ is \textbf{conformal covariant}, which means that for each $\mc H_i\in\Obj(\RepA)$, $\wtd I\in\Jtd$, $\xi\in\mc H_i(I)$, and $g\in\Gc$, there is a vector $g\xi g^{-1}\in\mc H_i(gI)$ such that
\begin{align}
L(g\xi g^{-1},g\wtd I)=gL(\xi,\wtd I)g^{-1},\qquad R(g\xi g^{-1},g\wtd I)=gR(\xi,\wtd I)g^{-1}\label{eq17}
\end{align}
hold when acting on any $\mc A$-module. (Cf. \cite[Thm. 3.13]{Gui21a}.)

Now we do not assume $\mc A$ to be conformal.  

By \cite[Rem. 2.2]{Gui21b}, for any $\wtd I=(I,\arg_I)\in\Jtd$, $x\in\mc A(I)$, and any $\mc A$-module $\mc H_i$,
\begin{align}
L(x\Omega,\wtd I)|_{\mc H_i}=R(x\Omega,\wtd I)|_{\mc H_i}=\pi_{i,I}(x).	\label{eq5}
\end{align}
Moreover, if $\mc H_j$ is an $\mc A$-module, and if $\xi\in\mc H_i(I),\eta\in\mc H_j(J)$, then (by locality and state-field correspondence)
\begin{align}
	L(\xi,\wtd I)\eta=R(\eta,\wtd J)\xi.\label{eq25}
\end{align}

The \textbf{functoriality} of $\scr E$ can be generalized to that for every $F\in\Hom_{\mc A}(\mc H_i,\mc H_{i'})$, $G\in\Hom_{\mc A}(\mc H_j,\mc H_{j'})$,  $\xi\in\mc H_i(I)$, and $\eta\in\mc H_j$, then
\begin{align}
	(F\boxtimes G)L(\xi,\wtd I)\eta=L(F\xi,\wtd I)G\eta,\qquad (G\boxtimes F)R(\xi,\wtd I)\eta=R(F\xi,\wtd I)G\eta.\label{eq2}
\end{align}
See \cite[Sec. 2]{Gui21b}. The following (adjoint) fusion relations were  proved in \cite[Prop. 2.3]{Gui21b}.

\begin{pp}\label{lb1}
Let $\mc H_i,\mc H_j,\mc H_k\in\Obj(\RepA)$, $\wtd I\in\Jtd$, and $\xi\in\mc H_i(I)$.\\
(a) If $\eta\in\mc H_j(I)$, then $L(\xi,\wtd I)\eta\in(\mc H_i\boxtimes\mc H_j)(I)$, $R(\xi,\wtd I)\eta\in(\mc H_j\boxtimes\mc H_i)(I)$, and
\begin{gather}
L(\xi,\wtd I)L(\eta,\wtd I)|_{\mc H_k}=L(L(\xi,\wtd I)\eta,\wtd I)|_{\mc H_k},\\
R(\xi,\wtd I)R(\eta,\wtd I)|_{\mc H_k}=R(R(\xi,\wtd I)\eta,\wtd I)|_{\mc H_k}.
\end{gather}
(b) If $\psi\in(\mc H_i\boxtimes H_j)(I)$ and $\phi\in (\mc H_j\boxtimes H_i)(I)$, then $L(\xi,\wtd I)^*\psi\in\mc H_j(I)$, $R(\xi,\wtd I)^*\phi\in\mc H_j(I)$, and
\begin{gather}
L(\xi,\wtd I)^*L(\psi,\wtd I)|_{\mc H_k}=L(L(\xi,\wtd I)^*\psi,\wtd I)|_{\mc H_k},\\
R(\xi,\wtd I)^*R(\phi,\wtd I)|_{\mc H_k}=R(R(\xi,\wtd I)^*\phi,\wtd I)|_{\mc H_k}.
\end{gather}

As a special case, when $x\in\mc A(I)$ and $\xi\in\mc H_i(I)$, then when acting on any $\mc H_j$,
\begin{align}
L(x\xi,\wtd I)=xL(\xi,\wtd I),\qquad R(x\xi,\wtd I)=xR(\xi,\wtd I).\label{eq11}	
\end{align}
\end{pp}

\begin{lm}\label{lb4}
For each $\mc H_i,\mc H_j\in\Obj(\RepA)$, if $\xi_1,\xi_2\in\mc H_i(I)$, then $L(\xi_1,\wtd I)^*L(\xi_2,\wtd I)|_{\mc H_0}\in\mc A(I)$, and
\begin{align}
L(\xi_1,\wtd I)^*L(\xi_2,\wtd I)|_{\mc H_j}=\pi_{j,I}\big(L(\xi_1,\wtd I)^*L(\xi_2,\wtd I)|_{\mc H_0}\big).	\label{eq8}
\end{align}
\end{lm}

\begin{proof}
Since $L(\xi_1,\wtd I)^*L(\xi_2,\wtd I)|_{\mc H_0}$ commutes with $\mc A(I')$, it  is in $\mc A(I)$. By the Prop. \ref{lb1} and \eqref{eq5}, 
\begin{align*}
&L(\xi_1,\wtd I)^*L(\xi_2,\wtd I)|_{\mc H_j}=L(L(\xi_1,\wtd I)^*\xi_2,\wtd I)|_{\mc H_j}\\	=&L(L(\xi_1,\wtd I)^*L(\xi_2,\wtd I)\Omega,\wtd I)|_{\mc H_j}=\pi_{j,I}\big(L(\xi_1,\wtd I)^*L(\xi_2,\wtd I)|_{\mc H_0}\big).	
\end{align*}
\end{proof}

The following lemma is \cite[Lemma 6.1]{Gui21a}.
\begin{lm}\label{lb2}
Suppose $\xi\in\mc H_i(I)$ and $L(\xi,\wtd I)|_{\mc H_0}=R(\xi,\wtd I)|_{\mc H_0}$ is a unitary map from $\mc H_0$ to $\mc H_i$, then for each $\mc A$-module $\mc H_j$, $L(\xi,\wtd I)|_{\mc H_j}$ are unitary maps from $\mc H_j$ to $\mc H_i\boxtimes\mc H_j$ and $\mc H_j\boxtimes\mc H_i$ respectively. In this case, we say $\xi$ is \textbf{unitary}.
\end{lm}

Note that unitary vectors in $\mc H_i(I)$ always exist since $\mc H_i$ and $\mc H_0$ are equivalent representations of the  type III factor $\mc A(I')$.

For each $t\in\Rbb$, let $\varrho(t)\in\UPSU$ be the anticlockwise rotation by $t$. The following property will not be used until Sec. \ref{lb19}.

\begin{lm}\label{lb20}
For each $\mc A$-module $\mc H_i,\mc H_j$, each $\wtd I\in\Jtd$  and each $\xi\in\mc H_i(I),\eta\in\mc H_j$,
\begin{align}
L(\xi,\varrho(2\pi)\wtd I)\eta=\ss_{j,i}\ss_{i,j}L(\xi,\wtd I)\eta.	
\end{align}
\end{lm}

Note that $(\varrho(2\pi)\wtd I)''=\wtd I$.

\begin{proof}
Let $\wtd I_1=\varrho(2\pi)\wtd I$. Assume without loss of generality that $\eta\in\mc H_i(I')=\mc H_i(I_1')$. Then by \eqref{eq25} and the braiding axiom \eqref{eq24},
\begin{align*}
&L(\xi,\wtd I_1)\eta=R(\eta,\wtd I_1')\xi=\ss_{j,i}L(\eta,\wtd I_1')\xi=\ss_{j,i}R(\xi,\wtd I_1'')\eta\\
=&\ss_{j,i}R(\xi,\wtd I)\eta=\ss_{j,i}\ss_{i,j}L(\xi,\wtd I)\eta.
\end{align*}
\end{proof}

Recall that an $\mc A$-module $\mc H_i$ is \textbf{dualizable} if there is also a module $\mc H_{\ovl i}$ and homomorphisms $\ev_{\ovl i,i}:\mc H_{\ovl i}\boxtimes\mc H_i\rightarrow\mc H_0,\ev_{i,\ovl i}:\mc H_i\boxtimes\mc H_{\ovl i}\rightarrow\mc H_0$ satisfying the conjugate equations
\begin{align}\label{eq22}
(\id_i\boxtimes\ev_{\ovl i,i})((\ev_{i,\ovl i})^*\boxtimes\id_i)=\id_i,\qquad(\id_{\ovl i}\boxtimes\ev_{i,\ovl i})((\ev_{\ovl i,i})^*\boxtimes\id_{\ovl i})=\id_{\ovl i}	
\end{align}
We refer the readers to \cite{LR97,Yam04,BDH14} for general results on dualizable objects in $C^*$-tensor categories. Recall  that the spaces of morphisms between dualizable objects are finite dimensional. We let
\begin{align*}
\RepfA=\text{the category of dualizable M\"obius covariant $\mc A$-modules}.	
\end{align*}
By \cite[Prop. 2.2]{GL96}, for each $\mc H_i\in\Obj(\RepfA)$ there is a unique strongly continuous unitary representation of $\UPSU$ on $\mc H_i$ making the $\mc A$-module $\mc H_i$ M\"obius covariant. From this uniqueness, it is easy to see that the morphisms of M\"obius covariant $\mc A$-modules intertwine the actions of $\UPSU$ (cf. \cite[Lemma B.1]{Gui21b}). By \cite[Sec. B]{Gui21b}, $\RepfA$ is closed under $\boxtimes$ and hence is a full braided $C^*$-tensor subcategory of $\RepA$, and the restriction of $\scr E$ to $\RepfA$ is \textbf{M\"obius covariant}, which means that  for each $\mc H_i\in\Obj(\RepfA)$, $\wtd I\in\Jtd$, $\xi\in\mc H_i(I)$, and $g\in\UPSU$,
\begin{align}
	L(g\xi,g\wtd I)=gL(\xi,\wtd I)g^{-1},\qquad R(g\xi,g\wtd I)=gR(\xi,\wtd I)g^{-1}\label{eq31}	
\end{align}
when acting on any object of $\RepfA$.



\section{$C^*$-Frobenius algebras and finite index extensions}

Recall that  $\mc A$ is always assumed to be an irreducible M\"obius net.

\begin{df}
	A \textbf{$C^*$-Frobenius algebra} in $\RepA$ is a triple $Q=(\mc H_a,\mu,\iota)$ where $\mc H_a$ is an $\mc A$-module, $\mu\in\Hom_{\mc A}(\mc H_a\boxtimes\mc H_a,\mc H_a)$, and $\iota\in\Hom_{\mc A}(\mc H_0,\mc H_a)$  satisfying the following conditions
	\begin{itemize}
		\item (Unit) $\mu(\iota\boxtimes\id_a)=\id_a=\mu(\id_a\boxtimes\iota)$.
		\item (Associativity+Frobenius relation) The following diagram commutes adjointly.
		\begin{equation}
			\begin{tikzcd}
				\mc H_a\boxtimes\mc H_a\boxtimes \mc H_a\arrow[rr,"\id_a\boxtimes\mu"]\arrow[d,"\mu\boxtimes\id_a"'] && \mc H_a\boxtimes\mc H_a\arrow[d,"\mu"]\\
				\quad\mc H_a\boxtimes\mc H_a\quad\arrow[rr,"\mu"] && \quad\mc H_a\quad
			\end{tikzcd}	
		\end{equation} 
	\end{itemize}
	If, moreover, $\iota$ is an isometry (i.e. $\iota^*\iota=\id_0$), we say $Q$ is \textbf{normalized}.
\end{df}

$\mc H_a$ is necessarily dualizable, since we can choose dual object $\mc H_{\ovl a}=\mc H_a$ and evaluations $\ev_{a,\ovl a}=\ev_{\ovl a,a}=\iota^*\mu$.

\begin{df}
	A (unitary) \textbf{left $Q$-module} means $(\mc H_i,\mu_i)$ where $\mc H_i$ is an $\mc A$-module, $\mu_i\in\Hom_{\mc A}(\mc H_a\boxtimes\mc H_i,\mc H_i)$, and the following are satisfied:	
	\begin{itemize}
		\item (Unit) $\mu_i(\iota\boxtimes\id_i)=\id_i$.
		\item (Associativity+Frobenius relation) The following diagram commutes adjointly.
		\begin{equation}\label{eq3}
			\begin{tikzcd}
				\mc H_a\boxtimes\mc H_a\boxtimes \mc H_i\arrow[rr,"\id_a\boxtimes\mu_i"]\arrow[d,"\mu\boxtimes\id_i"'] && \mc H_a\boxtimes\mc H_i\arrow[d,"\mu_i"]\\
				\quad\mc H_a\boxtimes\mc H_i\quad\arrow[rr,"\mu_i"] && \quad\mc H_i\quad
			\end{tikzcd}	
		\end{equation} 
	\end{itemize}	
\end{df}	

Assuming the unit property and that \eqref{eq3} commutes, then it is not hard to check that \eqref{eq3} commutes adjointly if and only if $\mu_i^*=(\id_a\boxtimes\mu_i)(\coev_{a,a}\boxtimes\id_i)$ where $\coev_{a,a}=\mu^*\iota$. Thus, the definition here agrees with the usual one.

\begin{df}\label{lb16}
A \textbf{(non-local) extension} of $\mc A$ denotes $(\mc H_a,\mc B,\iota)$, or simply $\mc B$, where $\mc H_a$ is an $\mc A$-module, $\iota\in\Hom_{\mc A}(\mc H_0,\mc H_a)$, and $\mc B$ associates to each $\wtd I\in\Jtd$ a von Neumann algebra $\mc B(\wtd I)$ on $\mc H_a$ such that the following conditions hold.
\begin{itemize}
\item (Extension property) For each $\wtd I=(I,\arg_I)\in\Jtd$, we have $\pi_{a,I}(\mc A(I))\subset\mc B(\wtd I)$.
\item (Isotony) If $\wtd I\subset\wtd J$ then $\mc B(\wtd I)\subset\mc B(\wtd J)$.
\item (Reeh-Schlieder property) For each $\wtd I$, $\iota\Omega$ is a cyclic separating vector for $\mc B(\wtd I)$.
\item (Relative locality) For each $\wtd I\in\Jtd$, $[\mc B(\wtd I),\pi_{a,I'}(\mc A(I'))]=0$.
\end{itemize}
\end{df}

	
\begin{df}\label{lb7}
If $\mc B$ is a non-local extension as above, a \textbf{$(\mc B,\mc A)$-module} is $(\pi_i,\mc H_i)$ where $\mc H_i$ is a $\mc A$-module, and $\pi_i$ associates to each $\wtd I\in\mc J$ a (normal) representation $\pi_{i,\wtd I}$ of $\mc B(\wtd I)$ on $\mc H_i$ satisfying the following conditions.
\begin{itemize}
\item (Extension property) For each $\wtd I\in\Jtd$, when acting on $\mc A(I)$,
\begin{align}
\pi_{i,I}=\pi_{i,\wtd I}\circ\pi_{a,I}\label{eq6}	
\end{align}
\item (Compatibility) If $\wtd I\subset\wtd J$ then $\pi_{i,\wtd J}|_{\mc B(\wtd I)}=\pi_{i,\wtd I}$.
\end{itemize}
\end{df}

Clearly $\mc H_a$ itself is a $(\mc B,\mc A)$-module, called the \textbf{vacuum $(\mc B,\mc A)$-module}.

\begin{rem}\label{lb17}
Any $(\mc B,\mc A)$-module is automatically \textbf{relatively local}, which means that for each $J\in\mc J$ (without arg value) disjoint from $\wtd I$, $[\pi_{i,\wtd I}(\mc B(\wtd I)),\pi_{i,J}(\mc A(J))]=0$. To see this, assume without loss of generality that $\wtd I$ and $J$ are contained in $\wtd K\in\Jtd$ (since $\mc A$ is additive). Then $[\pi_{i,\wtd I}(\mc B(\wtd I)),\pi_{i,J}(\mc A(J))]=[\pi_{i,\wtd K}(\mc B(\wtd I)),\pi_{i,K}(\mc A(J))]=\pi_{i,\wtd K}([\mc B(\wtd I),\pi_{a,K}(\mc A(J))])$ equals $0$ by the relative locality of $\mc B$.
\end{rem}

\begin{thm}\label{lb5}
If $Q=(\mc H_a,\mu,\iota)$ is a $C^*$-Frobenius algebra, then $(\mc H_a,\iota)$ can be equipped with a (necessarily unique)  extension $(\mc H_a,\mc B_Q,\iota)$ of $\mc A$ such that for each $\wtd I\in\Jtd$,
\begin{align}
\mc B_Q(\wtd I)=\{\mu L(\xi,\wtd I)|_{\mc H_a}:\xi\in\mc H_a(I)\}.	\label{eq28}
\end{align}
Moreover, $\mc B_Q$ satisfies the following properties:
\begin{itemize}
\item The commutant of $\mc B_Q(\wtd I)$ is
\begin{align}
\mc B_Q(\wtd I)'=\{\mu R(\eta,\wtd I')|_{\mc H_a}:\eta\in\mc H_a(I')\}\label{eq20}	
\end{align}
where $\wtd I'$ is the clockwise complement of $\wtd I$.
\item If $\mc H_a$ is M\"obius covariant, then $\mc B_Q$ is \textbf{M\"obius covariant}, which means that for each $g\in\UPSU$ and $\wtd I\in\Jtd$, we have
\begin{align}
g\mc B_Q(\wtd I)g^{-1}=\mc B_Q(g\wtd I).\label{eq16}	
\end{align} 
\item If $\mc A$ is a conformal net, then $\mc B_Q$ is \textbf{conformal covariant}, which means that \eqref{eq16} holds for each $g\in\Gc$ and $\wtd I\in\Jtd$.
\end{itemize}
\end{thm}

\begin{proof}
That $\mc B_Q$ is an extension is due to Thm. 4.7-(a,b,c,d) of \cite{Gui21b}. Formula \eqref{eq20} is due to \cite[Prop. 4.5]{Gui21b}. Note that although in \cite{Gui21b} we only considered M\"obius covariant modules, the proof of these results do not rely on M\"obius covariance. The key idea is to show for each fixed $\wtd I$ that  (1) $\mu L(\xi,\wtd I)|_{\mc H_a}$ commutes adjointly with $\mu R(\eta,\wtd I')|_{\mc H_a}$ for all $\xi\in\mc H_a(I)$ and $\eta\in\mc H_a(I')$ by diagram \eqref{eq27} (setting $\mc H_i=\mc H_a$ and $\mu_i=\mu$), and (2) any operator $X$ (resp. $Y$) on $\mc H_a$ commuting with all $\mu R(\eta,\wtd I')|_{\mc H_a}$ (resp. all $\mu L(\xi,\wtd I)|_{\mc H_a}$) satisfies $X\iota\Omega\in\mc H_a(I)$ and $X=\mu L(X\iota\Omega,\wtd I)|_{\mc H_a}$ (resp. $Y\iota\Omega\in\mc H_a(I')$ and $Y=\mu R(Y\iota\Omega,\wtd I')$). See \cite[Prop. 4.3, 4.5]{Gui21b}.

The proof of M\"obius or conformal covariance follows from that of the categorical extension $\scr E$ (see \eqref{eq16} and \eqref{eq17}) and the fact that the morphism $\mu$ intertwines the actions of $\UPSU$ or $\Gc$. 
\end{proof}

\begin{rem}\label{lb27}
Note that the von Neumann algebra $\mc B_Q(\wtd I)$ is not just generated by all $\mu L(\xi,\wtd I)|_{\mc H_a}$. It is exactly the set of all such operators. Also, for each $X\in\mc B_Q(\wtd I)$ and $\xi\in\mc H_a(I)$, 
\begin{align}
X=\mu L(\xi,\wtd I)|_{\mc H_a}\quad\Leftrightarrow\quad\xi= X\iota\Omega.\label{eq32}	
\end{align}
Indeed, assume the left, then $X\iota\Omega=\mu L(\xi,\wtd I)\iota\Omega=\mu(\id_a\boxtimes\iota)L(\xi,\wtd I)\Omega=\xi$. Assume the right, then $X$ and $\mu L(\xi,\wtd I)|_{\mc H_a}$ both send $\iota\Omega$ to $\xi$. So they are equal since $\iota\Omega$ is separating for $\mc B_Q(\wtd I)$.

From this observation, we see that $Q$ is uniquely determined by $\mc B_Q$ since $\mu$ must send each $L(X\iota\Omega,\wtd I)\eta$ (where $X\in\mc B_Q(\wtd I),\eta\in\mc H_a$) to $X\eta$.
\end{rem}

\begin{rem}
Since $\pi_{a,I}(\mc A(I))\subset\mc B_Q(\wtd I)$, for each $x\in\mc A(I)$, $\pi_{a,I}(x)$ can be written in the form $\mu L(\xi,\wtd I)|_{\mc H_a}$. In fact, as $\pi_{a,I}(x)\iota\Omega=\iota x\Omega\in\mc H_a(I)$, by \eqref{eq32} we have
\begin{align}
	\pi_{a,I}(x)=\mu L(\iota x\Omega,\wtd I)|_{\mc H_a}.\label{eq7}	
\end{align}
\end{rem}

We recall the following well known facts about von Neumann algebra representations.
\begin{lm}\label{lb3}
Let $\fk S$ be a set. Let $(x_s)_{s\in\fk S}$ and $(y_s)_{s\in\fk S}$ be collections (labeled by $\fk S$) of bounded linear operators on Hilbert spaces $\mc H$ and $\mc K$ respectively. Let $M$ be the von Neumann algebra on $\mc H$ generated by $\{x_s,x_s^*:s\in \fk S\}$. Suppose that there is a set $\fk T$ of bounded  linear maps from $\mc H$ to $\mc K$ such that $\Span_{T\in\fk T}(T\mc H)$ is dense in $\mc K$, and that for each $s\in \fk S$ and $T\in\fk T$ we have
	\begin{align*}
		Tx_s=y_sT,\quad Tx_s^*=y_s^*T.	
	\end{align*}
	Then there is a unique  (normal) representation $\pi$ of $M$ on $\mc K$ satisfying $\pi(x_s)=y_s$ for each $s\in\fk S$.
\end{lm}

Consequently, if $x_{s_1}=x_{s_2}$, then we must have $y_{s_1}=y_{s_2}$.

\begin{proof}
	Uniqueness is obvious. As for the existence, let $N$ be the  commutant of $\{y_s,y_s^*:s\in\fk S\}$. By Zorn's lemma, there is a maximal set $\fk E$ of mutually orthogonal projections in $N$ such that for each $e\in\fk E$ there is a partial isometry $U_e:\mc H\rightarrow\mc K$ such that $U_eU_e^*=e$, and that $U_ex_s=y_sU_e,U_ex_s^*=y_s^*U_e$ for each $s\in\fk S$.  Let $E=\sum_{e\in\fk E}e$. We claim that $E=\id_{\mc K}$. Then $\pi(x)=\sum_{e\in\fk E}U_exU_e^*$ ($x\in M$) is the desired representation. 
	
	If $E\neq \id_{\mc K}$, then there is $T\in\fk T$ such that $T_1:=(\id-E)T\neq 0$. Note that the following diagram commutes adjointly
	\begin{equation*}
		\begin{tikzcd}
			\mc H\arrow[r,"x_s"]\arrow[d,"T_1"']&\mc H\arrow[d,"T_1"]\\
			\mc K\arrow[r,"y_s"]&\mc K
		\end{tikzcd}	
	\end{equation*}
	Thus, if we polar-decompose $T_1$ as $T_1=U_1H_1$ where $U_1$ is the  partial isometry, then the above diagram also commutes adjointly if $T_1$ is replaced by $U_1$. Then the set $\fk S\cup\{U_1U_1^*\}$ is larger than $\fk S$ but satisfies the condition described in the first paragraph. This is a contradiction.
\end{proof}

The following theorem is the Connes-fusion version of \cite[Lem. 3.1]{EP03} and \cite[Prop. 3.24]{BKLR15}.

\begin{thm}\label{lb6}
If $(\mc H_i,\mu_i)$ is a left $Q$-module, then the $\mc A$-module $\mc H_i$ can be equipped with a  (necessarily uniquely) $(\mc B_Q,\mc A)$-module structure $(\mc H_i,\pi_i)$ such that for all $\wtd I\in\Jtd,\xi\in\mc H_a(I)$,
\begin{align}\label{eq4}
\pi_{i,\wtd I}\big(\mu L(\xi,\wtd I)|_{\mc H_a}\big)=\mu_i L(\xi,\wtd I)|_{\mc H_i}.
\end{align}

Conversely, if $(\mc H_i,\pi_i)$ is a $(\mc B_Q,\mc A)$-module, then the associated $\mc A$-module $\mc H_i$ is equipped with a (necessarily unique) left $Q$-module structure $(\mc H_i,\mu_i)$ which gives rise to $(\mc H_i,\pi_i)$ via the relation \eqref{eq4}.
\end{thm}

We divide the proof into two parts.

\begin{proof}[Part 1 of the proof]
We choose a left $Q$-module $(\mc H_i,\mu_i)$ and construct the $(\mc B_Q,\mc A)$-module $(\mc H_i,\pi_i)$. Choose $\wtd J\in\Jtd$ clockwise to $\wtd I$ and $\eta\in\mc H_i(J)$.
\begin{equation}\label{eq27}
\begin{tikzcd}
\mc H_a\arrow[rr,"{R(\eta,\wtd J)}"]\arrow[d,"{L(\xi,\wtd I)}"'] &&\mc H_a\boxtimes\mc H_i\arrow[rr,"\mu_i"]\arrow[d,"{L(\xi,\wtd I)}"']&&\mc H_i\arrow[d,"{L(\xi,\wtd I)}"']\\
\mc H_a\boxtimes\mc H_a\arrow[rr,"{R(\eta,\wtd J)}"]\arrow[d,"\mu"']&&\mc H_a\boxtimes\mc H_a\boxtimes\mc H_i\arrow[rr,"\id_a\boxtimes\mu_i"]\arrow[d,"\mu\boxtimes\id_i"']&&\mc H_a\boxtimes\mc H_i\arrow[d,"\mu_i"']\\
\mc H_a\arrow[rr,"{R(\eta,\wtd J)}"]&&\mc H_a\boxtimes\mc H_i\arrow[rr,"\mu_i"]&&\mc H_i
\end{tikzcd}
\end{equation}
By the locality and the functorality of $\scr E$ (as well as $(F\boxtimes G)^*=F^*\boxtimes G^*$), together with the associativity and Frobenius relation for $\mu_i$, each of the four small diagrams commutes adjointly. Thus, the largest diagram commutes adjointly. Note that $\mu_i$ is surjective since $\mu_i(\iota\boxtimes\id_i)=\id_i$. Thus, if we choose $\eta$ to be unitary, then $\mu_i R(\eta,\wtd J):\mc H_a\rightarrow\mc H_i$ is surjective. Therefore, by Lemma \ref{lb3}, there is a unique representation $\pi_{i,\wtd I}$ of $\mc B_Q(\wtd I)$ on $\mc H_a$ such that \eqref{eq4} holds. The compatibility condition is easy to check.

It remains to check the extension property \eqref{eq6} on $\mc A(I)$. Choose any $x\in\mc A(I)$ and recall \eqref{eq7}. Then
\begin{align}
&\pi_{i,\wtd I}\circ\pi_{a,I}(x)=\pi_{i,\wtd I}(\mu L(\iota x\Omega,\wtd I)|_{\mc H_a})=\mu_i L(\iota x\Omega,\wtd I)|_{\mc H_i}\nonumber\\
=&\mu_i(\iota\boxtimes \id_i) L(x\Omega,\wtd I)|_{\mc H_i}=L(x\Omega,\wtd I)|_{\mc H_i}=\pi_{i,I}(x).\label{eq12}
\end{align}
\end{proof}

\begin{proof}[Part 2 of the proof]
We choose a $(\mc B_Q,\mc A)$-module $(\mc H_i,\pi_i)$ and define the left $Q$-module $(\mc H_i,\mu_i)$.

Step 1. Let $\lambda\geq 0$ be $\lVert\mu\lVert^2$, the square operator norm of $\mu$. Choose any $\wtd I\in\Jtd$. Then for any $\xi_1,\dots,\xi_N\in\mc H_i(I)$, we have the following Pimsner-Popa inequality
\begin{align}
\Big[L(\xi_k,\wtd I)^*\mu^*\mu L(\xi_l,\wtd I)\big|_{\mc H_a}\Big]_{k,l}\leq \lambda \Big[\pi_{a,I}\big(L(\xi_k,\wtd I)^*L(\xi_l,\wtd I)\big|_{\mc H_0}\big)\Big]_{k,l}	\label{eq9}
\end{align}
for the elements of $\mc B_Q(\wtd I)\otimes \End(\Cbb^N)$. Indeed, choose any $\psi_\blt=(\psi_1,\dots,\psi_N)\in\mc H_a\otimes\Cbb^N$, then by Lemma \ref{lb4},
\begin{align*}
&\big\langle \big[L(\xi_k,\wtd I)^*\mu^*\mu L(\xi_l,\wtd I)|_{\mc H_a}\big]_{k,l}\cdot \psi_\blt\big|\psi_\blt \big\rangle	=\big\lVert \sum_{l=1}^N \mu L(\xi_l,\wtd I)\psi_l \big\lVert^2\\
\leq&\lambda\big\lVert \sum_{l=1}^N L(\xi_l,\wtd I)\psi_l \big\lVert^2=\lambda\big\langle \big[L(\xi_k,\wtd I)^*L(\xi_l,\wtd I)|_{\mc H_a}\big]_{k,l}\cdot \psi_\blt\big|\psi_\blt \big\rangle\\
\xlongequal{\eqref{eq8}}&\lambda\big\langle \big[\pi_{a,I}(L(\xi_k,\wtd I)^*L(\xi_l,\wtd I)|_{\mc H_0})\big]_{k,l}\cdot \psi_\blt\big|\psi_\blt \big\rangle.
\end{align*}

Apply $\pi_{i,\wtd I}\otimes\id_{\Cbb^N}$ to both sides of \eqref{eq9} and notice \eqref{eq6}, we get
\begin{align*}
\Big[\pi_{i,\wtd I}\big(L(\xi_k,\wtd I)^*\mu^*\mu L(\xi_l,\wtd I)\big|_{\mc H_a}\big)\Big]_{k,l}\leq \lambda \Big[\pi_{i,I}\big(L(\xi_k,\wtd I)^*L(\xi_l,\wtd I)\big|_{\mc H_0}\big)\Big]_{k,l}.
\end{align*}
By \eqref{eq8} again, this is equivalent to
\begin{align}
	\Big[\pi_{i,\wtd I}\big(\mu L(\xi_k,\wtd I)\big|_{\mc H_a}\big)^*\pi_{i,\wtd I}\big(\mu L(\xi_l,\wtd I)\big|_{\mc H_a}\big)\Big]_{k,l}\leq \lambda \Big[L(\xi_k,\wtd I)^*L(\xi_l,\wtd I)\big|_{\mc H_i}\Big]_{k,l}.
\end{align}
Therefore, we have a bounded linear map
\begin{gather}\label{eq10}
\begin{gathered}
\mu_{i,\wtd I}:\mc H_a\boxtimes\mc H_i\rightarrow\mc H_i\\
L(\xi,\wtd I)\eta\mapsto \pi_{i,\wtd I}\big(\mu L(\xi,\wtd I)\big|_{\mc H_a}\big)\eta
\end{gathered}	
\end{gather}
($\forall\xi\in\mc H_a(I),\eta\in\mc H_j$) with norm $\leq \sqrt{\lambda}=\lVert\mu\lVert$.

Step 2. Clearly $\mu_{i,\wtd I}=\mu_{i,\wtd J}$ when $\wtd I\subset \wtd J$, hence when $\wtd K\subset\wtd I,\wtd J$ for some $\wtd K$, and hence for all $\wtd I,\wtd J$. Therefore the map \eqref{eq10} is independent of  $\wtd I$ and hence can be written as $\mu_{i}$.

To see that $\mu_i$ is an $\mc A$-module homomorphism, we choose any $x\in\mc A(I)$, and note that by \eqref{eq11} and the fact that $\mu$ is an $\mc A$-module homomorphism, $\mu_i$ sends $xL(\xi,\wtd I)\eta=L(x\xi,\wtd I)\eta$ to
\begin{align*}
&\pi_{i,\wtd I}(\mu L(x\xi,\wtd I)|_{\mc H_a})\eta=\pi_{i,\wtd I}(\mu xL(\xi,\wtd I)|_{\mc H_a})\eta=\pi_{i,\wtd I}(\pi_{a,I}(x)\mu L(\xi,\wtd I)|_{\mc H_a})\eta\\
\xlongequal{\eqref{eq6}}&\pi_{i,I}(x)\pi_{i,\wtd I}(\mu L(\xi,\wtd I)|_{\mc H_a})\eta=\pi_{i,I}(x)\mu_i(L(\xi,\wtd I)\eta).
\end{align*}

To check that $\mu_i$ satisfies the unit property, note that by the argument in \eqref{eq12}, we have
\begin{align*}
\pi_{i,\wtd I}\circ\pi_{a,I}(x)=\mu_i(\iota\boxtimes \id_i) L(x\Omega,\wtd I)|_{\mc H_i},\qquad \pi_{i,I}(x)=L(x\Omega,\wtd I)|_{\mc H_i}.
\end{align*}
Since $\pi_{i,\wtd I}\circ\pi_{a,I}(\id)=\pi_{i,I}(\id)=\id_i$, we obtain $\mu_i(\iota\boxtimes \id_i)=\id_i$.

Finally, we check the associativity and the Frobenius relation. Choose any $\xi_1,\xi_2\in\mc H_a(I)$ and $\eta\in\mc H_i$. Then
\begin{align*}
&\pi_{i,\wtd I}(\mu L(\xi_1,\wtd I)|_{\mc H_a})\pi_{i,\wtd I}(\mu L(\xi_2,\wtd I)|_{\mc H_a})\eta=\mu_i L(\xi_1,\wtd I)\mu_i L(\xi_2,\wtd I)\eta\\
=&\mu_i(\id_a\boxtimes\mu_i)L(\xi_1,\wtd I)L(\xi_2,\wtd I)\eta	
\end{align*}
equals
\begin{align*}
&\pi_{i,\wtd I}(\mu L(\xi_1,\wtd I)\mu L(\xi_2,\wtd I)|_{\mc H_0})\eta=\pi_{i,\wtd I}(\mu(\id_a\boxtimes\mu) L(\xi_1,\wtd I) L(\xi_2,\wtd I)|_{\mc H_0})\eta\\
=& \pi_{i,\wtd I}(\mu(\mu\boxtimes\id_a) L(L(\xi_1,\wtd I)\xi_2,\wtd I)|_{\mc H_0})	=\pi_{i,\wtd I}(\mu L(\mu L(\xi_1,\wtd I)\xi_2,\wtd I)|_{\mc H_0})\eta\\
=&\mu_i L(\mu L(\xi_1,\wtd I)\xi_2,\wtd I)\eta=\mu_i(\mu\boxtimes \id_i) L(L(\xi_1,\wtd I)\xi_2,\wtd I)\eta\\
=&\mu_i(\mu\boxtimes \id_i) L(\xi_1,\wtd I)L(\xi_2,\wtd I)\eta
\end{align*}
where the naturality \eqref{eq2} is used many times. Thus, by the density of fusion product, we must have $\mu_i(\id_a\boxtimes\mu_i)=\mu_i(\mu\boxtimes\id_i)$.

We now know that in \eqref{eq27}, the lower right small diagram commutes, and the other three commute adjointly. Therefore, the largest diagram commutes. Namely, for $X=\mu L(\xi,\wtd I)|_{\mc H_a}$, we have $\mu_iR(\eta)\cdot  X=\pi_{i,\wtd I}(X)\cdot \mu_iR(\eta)$. By \eqref{eq28}, $X^*=\mu L(\xi',\wtd I)|_{\mc H_a}$ for some $\xi'\in\mc H_a(I)$. Therefore  $\mu_iR(\eta)\cdot X^*=\pi_{i,\wtd I}(X)^*\cdot \mu_iR(\eta)$, which is equivalent to that the largest diagram of \eqref{eq27} commutes adjointly. Therefore, in \eqref{eq27},  the two paths $\rightarrow \uparrow \rightarrow\uparrow$ and $\rightarrow \rightarrow \uparrow \uparrow$ from the lower left corner $\mc H_a$  to the upper right corner $\mc H_i$ are equal. Thus, by the density of fusion product, the lower right cell must commute adjointly. This proves the Frobenius relation.
\end{proof}

Let us formulate the above theorem in a more categorical way. Let $\RepL(Q)$ be the $C^*$-category of left $Q$-modules. If $(\mc H_i,\mu_i)$ and $(\mc H_j,\mu_j)$ are left $Q$-modules, then 
\begin{align*}
\HomL_Q(\mc H_i,\mc H_j)=\{T\in\Hom_{\mc A}(\mc H_i,\mc H_j):T\mu_i=\mu_j(\id_a\boxtimes T)\}.
\end{align*}
(It is easy to check that $T\in\HomL_Q(\mc H_i,\mc H_j)\Leftrightarrow T^*\in\HomL_Q(\mc H_j,\mc H_i)$.) If $\mc B$ is an extension of $\mc A$, we let $\Rep(\mc B,\mc A)$ be the $C^*$-category of $(\mc B,\mc A)$-modules such that for $(\mc B,\mc A)$-modules $(\mc H_i,\pi_i)$ and $(\mc H_j,\pi_j)$, the space of morphims is 
\begin{align*}
\Hom_{\mc B}(\mc H_i,\mc H_j)=\{T\in\Hom(\mc H_i,\mc H_j):T\pi_{i,\wtd I}(X)=\pi_{j,\wtd I}(X)T\text{ for all }\wtd I\in\Jtd,X\in\mc B(\wtd I)\}.	
\end{align*} 
It is clear that
\begin{align*}
\Hom_{\mc B}(\mc H_i,\mc H_j)\subset \Hom_{\mc A}(\mc H_i,\mc H_j).	
\end{align*}

\begin{Mthm}\label{lb12}
Let $Q=(\mc H_a,\mu,\iota)$ be a $C^*$-Frobenius algebra in $\RepA$. For each left $Q$-module $(\mc H_i,\mu_i)$, define the corresponding $(\mc B_Q,\mc A)$-module $(\mc H_i,\pi_i)$ as in Thm. \ref{lb6}. Then for any objects $\mc H_i,\mc H_j$ we have $\HomL_Q(\mc H_i,\mc H_j)=\Hom_{\mc B_Q}(\mc H_i,\mc H_j)$. Therefore, the $*$-functor 
\begin{gather*}
	\fk F:\RepL(Q)\rightarrow\Rep(\mc B_Q,\mc A)\\
	\left\{
\begin{array}{l}
(\mc H_i,\mu_i)\in\Obj(\RepL(Q))\mapsto (\mc H_i,\pi_i)\in\Obj(\Rep(\mc B_Q,\mc A)),\\[1ex]
T\in\HomL_Q(\mc H_i,\mc H_j)\mapsto  T\in\Hom_{\mc B_Q}(\mc H_i,\mc H_j)		
\end{array}	
	\right.
\end{gather*}
is an isomorphism of $C^*$-tensor categories.
\end{Mthm}

\begin{proof}
Choose any $T\in\Hom_{\mc A}(\mc H_i,\mc H_j)$. For each $X\in\mc B(\wtd I)$ written as $X=\mu L(\xi,\wtd I)|_{\mc H_a}$ where $\xi\in\mc H_a(I)$, we have (for any $\eta\in\mc H_i$)
\begin{align*}
T\pi_{i,\wtd I}(X)\eta=T\mu_i L(\xi,\wtd I)\eta	
\end{align*}
and
\begin{align*}
\pi_{j,\wtd I}(X)T\eta=\mu_i L(\xi,\wtd I)T\eta=\mu_i(\id_a\boxtimes T)L(\xi,\wtd I)\eta.
\end{align*}
So $T\in\HomL_Q(\mc H_i,\mc H_j)$ iff $T\in\Hom_{\mc B_Q}(\mc H_i,\mc H_j)$.
\end{proof}

\begin{rem}
The above proof actually shows that if $T$ intertwines $\mc A$ and $\mc B_Q(\wtd I)$ for some $\wtd I\in\Jtd$, then $T\in\HomL_Q(\mc H_i,\mc H_j)$. We conclude
\begin{align}
\HomL_Q(\mc H_i,\mc H_j)=\Hom_{\mc B_Q}(\mc H_i,\mc H_j)=\Hom_{\mc A}(\mc H_i,\mc H_j)\cap \Hom_{\mc B_Q(\wtd I)}(\mc H_i,\mc H_j).
\end{align}
Thus, if $\mc H_i,\mc H_j$ are both M\"obius covariant and dualizable as $\mc A$-modules, then for each $\wtd I$, as $\Hom_{\mc A}(\mc H_i,\mc H_j)=\Hom_{\mc A(I),\mc A(I')}(\mc H_i,\mc H_j)$ by \cite[Thm. 2.3]{GL96}, we see that $\Hom_{\mc B_Q}(\mc H_i,\mc H_j)$ is precisely the set of bounded operators intertwining the actions of $\mc B_Q(\wtd I)$ and $\mc A(I')$.
\end{rem}

We now study the question of when an extension arises from a $C^*$-Frobenius algebra. We first review some basic facts about finite index extensions of type III factors \cite{Lon89,Lon90,Kos98,BDH14}. Let $\mc N\subset\mc M$ be a pair of von Neumann algebras where $\mc N$ is a type III factor, and suppose that there is a (normal) faithful conditional expectation $\mc E:\mc M\rightarrow\mc N$. Then  $\mc E$ has finite index if and only if  there is $\lambda>0$ such that  for each $N>0$ and each $X_1,\dots,X_N\in\mc M$, we have the Pimsner-Popa inequality
\begin{align}
[X_k X_l^*]_{k,l}\leq \lambda [\mc E(X_k X_l^*)]_{k,l}	\label{eq13}
\end{align}
for the two elements of $\mc M\otimes\End(\Cbb^N)$. Note that if the finite index holds for one $\mc E$, then it holds for every faithful normal conditional expectation $\mc M\rightarrow\mc N$, cf. \cite[Prop. 5.4]{Lon89} and the paragraph thereafter. (See also Rem. \ref{lb15} for a related discussion.) In this case, we say $\mc N\subset\mc M$ has finite index. 

In the following theorem (which is the Connes-fusion version of (a variant of) \cite[Thm. 4.9]{LR95}), $\mc H_a$ is not assumed to be M\"obius covariant.

\begin{thm}\label{lb11}
Let $(\mc H_a,\mc B,\iota)$ be an extension of $\mc A$. Then the following are equivalent.
\begin{enumerate}[label=(\arabic*)]
\item $\mc H_a$ is a dualizable object in $\RepA$.
\item For each $I\in\mc J$, consider $\mc H_a$ as an $\mc A(I)-\mc A(I')^\opp$ bimodule. Then  $\mc H_a$ is dualizable in the $C^*$-tensor category of $\mc A(I)-\mc A(I')^\opp$ bimodules. Equivalently\footnote{See \cite{Lon90}, \cite[Sec. 2.7]{LR95}, or \cite[Sec. 7]{BDH14}}, the extension $\pi_{a,I}(\mc A(I))\subset \pi_{a,I}(\mc A(I'))'$ has finite index.
\item For each $\wtd I \in\Jtd$, $\pi_{a,I}(\mc A(I))\subset\mc B(\wtd I)$ is a finite index extension.
\item $\mc B=\mc B_Q$ for some $C^*$-Frobenius algebra $Q=(\mc H_a,\mu,\iota)$ in $\RepA$.
\end{enumerate} 	
Moreover, the $Q$ in (4) is unique.
\end{thm}	

We call any $\mc B$ satisfying one of the above equivalent conditions a \textbf{(non-local) finite index extension} of $\mc A$.

\begin{proof}
(1)$\Rightarrow$(2) since $\RepA$ is naturally a (non-necessarily full)\footnote{Note that $\RepfA$ is actually a full subcategory, due to \cite[Thm. 2.3]{GL96}. But we will not use this fact here.}  $C^*$-tensor subcategory of the category of $\mc A(I)-\mc A(I')^\opp$ bimodules. (2)$\Rightarrow$(3) is obvious. (4)$\Rightarrow$(1) since, as mentioned, any object with a $C^*$-Frobenius algebra structure is dualizable. So we only need to prove (3)$\Rightarrow$(4). Let us assume (3).

Step 1. Since $\mc A$ is irreducible, by scaling $\iota$, we may assume $\iota$ is an isometry. For each $\wtd I\in\Jtd$, by the Reeh-Schlieder property for $\mc B$,  $\mc B(\wtd I)\iota\Omega$ is dense in $\mc H_a$. Moreover, by the relative locality, for each $X\in\mc B(\wtd I)$,  $X\iota$ intertwines the actions of $\mc A(I')$. Hence $\mc B(\wtd I)\iota\Omega\subset\mc H_a(I)$.   Therefore, by the density of fusion product, all $L(X\iota\Omega)\eta=R(\eta)X\iota\Omega$ (where $X\in\mc B(\wtd I),\eta\in\mc H_a(I')$) form a dense subspace of $\mc H_a\boxtimes\mc H_a$. 

Note that $\iota\iota^*$ is the projection of $\mc H_a$ onto $\iota(\mc H_0)$. Thus it determines a faithful normal conditional expectation $\mc E_{\wtd I}:\mc B(\wtd I)\rightarrow\pi_{a,I}(\mc A(I))$ satisfying 
\begin{align*}
\iota\iota^* X\iota\iota^*=\mc E_{\wtd I}(X)\iota\iota^*	
\end{align*}
for each $X\in\mc B(\wtd I)$. Note that $X\iota=L(X\iota\Omega,\wtd I)$ since both sides intertwine $\mc A(I')$ and send $\Omega$ to $X\iota\Omega$. Thus, for each $N>0$ and $X_1,\dots,X_N\in\mc B(\wtd I)$, we have
\begin{align*}
\iota\iota^* X_k^*X_l\iota\iota^*=\iota L(X_k\iota\Omega,\wtd I)^* L(X_l\iota\Omega,\wtd I)\iota^*=\pi_{a,I}(L(X_k\iota\Omega,\wtd I)^* L(X_l\iota\Omega,\wtd I)|_{\mc H_0})\iota\iota^*.
\end{align*}
Thus
\begin{align}
\mc E_{\wtd I}(X_k^*X_l)=\pi_{a,I}(L(X_k\iota\Omega,\wtd I)^* L(X_l\iota\Omega,\wtd I)|_{\mc H_0})\xlongequal{\eqref{eq8}}L(X_k\iota\Omega,\wtd I)^* L(X_l\iota\Omega,\wtd I)|_{\mc H_a}.	
\end{align}
Choose $\lambda>0$ satisfying \eqref{eq13} for each $N$. Then
\begin{align}
\big[X_k^*X_l\big]_{k,l}\leq \lambda \big[L(X_k\iota\Omega,\wtd I)^* L(X_l\iota\Omega,\wtd I)\big|_{\mc H_a}\big]_{k,l}	
\end{align}
for the elements of $\mc B(\wtd I)\otimes\End(\Cbb^N)$. We conclude that there is a unique bounded linear operator 
\begin{align}
\mu_{\wtd I}:\mc H_a\boxtimes\mc H_a\rightarrow\mc H_a,\qquad 	L(X\iota\Omega,\wtd I)\eta\mapsto X\eta\label{eq15}
\end{align}
for each $X\in\mc B(\wtd I),\eta\in\mc H_a$. \\

Step 2. Similar to the proof of Thm. \ref{lb6}, $\mu_{\wtd I}$ is independent of $\wtd I$. So we write it as $\mu$. For each $x\in\mc A(I)$, $\mu xL(X\iota\Omega,\wtd I)\eta=\mu L(\pi_{a,I}(x) X\iota\Omega,\wtd I)\eta=\pi_{a,I}(x) X\eta=\pi_{a,I}(x)\mu L(X\iota\Omega,\wtd I)\eta$. Thus $\mu$ is an $\mc A$-module morphism. To check that $\mu$ satisfies the unit property, we choose any $\eta\in\mc H_a$ and compute
\begin{align*}
\mu(\iota\boxtimes\id_a)\eta\xlongequal{\eqref{eq5}}\mu(\iota\boxtimes\id_a)L(\Omega,\wtd I)\eta=\mu L(\id \iota\Omega,\wtd I)\eta=\id\eta=\eta,	
\end{align*}
and choose any $X\in\mc B(\wtd I)$ and compute
\begin{align*}
\mu (\id_a\boxtimes\iota)X\iota\Omega\xlongequal{\eqref{eq14}}\mu (\id_a\boxtimes\iota)L(X\iota\Omega,\wtd I)\Omega=\mu L(X\iota\Omega,\wtd I)\iota\Omega=X\iota\Omega.
\end{align*}

For each $X_1,X_2\in\mc B(\wtd I)$ and $\eta\in\mc H_a$, $X_1X_2\eta$ equals both
\begin{align*}
&\mu L(X_1X_2\iota\Omega,\wtd I)\eta=\mu L(\mu L(X_1\iota\Omega,\wtd I)X_2\iota\Omega,\wtd I)\eta\\
=&\mu(\mu\boxtimes\id_a) L(L(X_1\iota\Omega,\wtd I)X_2\iota\Omega,\wtd I)\eta=\mu(\mu\boxtimes\id_a) L(X_1\iota\Omega,\wtd I)L(X_2\iota\Omega,\wtd I)\eta	
\end{align*}
and
\begin{align*}
\mu L(X_1\iota\Omega,\wtd I)\mu L(X_2\iota\Omega,\wtd I)\eta	=\mu(\id_a\boxtimes\mu) L(X_1\iota\Omega,\wtd I)L(X_2\iota\Omega,\wtd I)\eta.
\end{align*}
This proves the associativity of $\mu$. The Frobenius relation can be proved in the same way as the last paragraph of the proof of Thm. \ref{lb6}. 

We have constructed a $C^*$-Frobenius algebra $Q=(\mc H_a,\mu,\iota)$. Then $\mc B_Q(\wtd I)$ consists of all $\mu L(\xi,\wtd I)|_{\mc H_a}$ where $\xi\in\mc H_a(I)$. Thus, it is generated by all $\mu L(X\iota\Omega,\wtd I)=X$ where $X\in\mc B(\wtd I)$. Therefore $\mc B=\mc B_Q$. Thus the map $Q\mapsto\mc B_Q$ is surjective. By Rem. \ref{lb27}, it is injective.
\end{proof}

\begin{rem}\label{lb14}
In this article, we construct extensions from $Q$ using $L$ operators. One can also use $R$ operators. Then one obtains a finite index extension $(\mc H_a,\mc B'_Q,\iota)$ where
\begin{align}
\mc B_Q'(\wtd I)=	\{\mu R(\eta,\wtd I)|_{\mc H_a}:\eta\in\mc H_a(I)\}.
\end{align}
(Note that $\mc B_Q'(\wtd I')=\mc B_Q(\wtd I)'$ by \eqref{eq20}.) Then similar results as Thm. \ref{lb6}, \ref{lb11}, and Main Thm.  \ref{lb12} hold for such extensions. For instance, the $C^*$-category of right $Q$-modules is canonically isomorphic to the $C^*$-category of $\mc B_Q'$-modules.
\end{rem}

\section{Isomorphisms of $C^*$-Frobenius algebras and extensions}

\begin{df}\label{lb8}
Let $Q^a=(\mc H_a,\mu^a,\iota^a)$ and $Q^b=(\mc H_b,\mu^b,\iota^b)$ be $C^*$-Frobenius algebras in $\RepA$. A \textbf{left isomorphism} $V:Q^a\rightarrow Q^b$ denotes  $V\in\Hom_{\mc A}(\mc H_a,\mc H_b)$ with bounded inverse $V^{-1}\in\Hom_{\mc A}(\mc H_b,\mc H_a)$ satisfying
\begin{gather}\label{eq18}
	\begin{gathered}
V\iota^a=\iota^b,\\
V\mu^a=\mu^b(V\boxtimes V),\\
V^*\mu^b(V\boxtimes \id_b)=\mu^a(\id_a\boxtimes V^*).
	\end{gathered}
\end{gather}
If $V$ is unitary, we say that $V$ is a \textbf{unitary isomorphism}.
\end{df}

\begin{rem}\label{lb9}
It is a routine check that $V^{-1}:Q^b\rightarrow Q^a$ is a left isomorphism, and that if $W:Q^b\rightarrow Q^c$ is a left isomorphism, then so is $WV:Q^a\rightarrow Q^c$. 
\end{rem}

\begin{rem}
Assuming the second equation of \eqref{eq18}, it is easy to see that the third one of \eqref{eq18} is equivalent to $V^*V\in\EndL_Q(\mc H_a)$, namely,
\begin{align}
V^*V\mu ^a=\mu^a(\id_a\boxtimes V^*V).	
\end{align}
Thus, our definition is in line with \cite[Def. 2.4]{NY18}.

From this observation, we see that if $\mc H_a,\mc H_b$  are irreducible $Q$-modules, then a left isomorphism $V:Q^a\rightarrow Q^b$ is unitary up  to scalar multiplication. It is unitary if we also have $(\iota^a)^*\iota^a=(\iota^b)^*\iota^b$. Thus, all left isomorphisms of irreducible normalized $C^*$-Frobenius algebras are unitary. 

More generally, by polar-decomposing $V$, we see that the  positive operators $H\in\EndL_Q(\mc H_a)$ with bounded inverse correspond surjectively to the unitary isomorphism classes of $C^*$-Frobenius algebras in $\Rep(\mc A)$ that are left isomorphic to $Q$. The correspondence is given by
\begin{align}
H\mapsto (\mc H_a,\mu^a(H^{-1}\boxtimes\id_a),H\iota^a).	
\end{align}
This relation is similar to that between the  faithful normal conditional expectations for a subfactor $\mc N\subset\mc M$ and the relative commutant $\mc N'\cap \mc M$ (cf. \cite{CD75} or \cite[Sec. A.3]{Kos98}). See Rem. \ref{lb15} for further discussions.
\end{rem}

\begin{rem}
It was shown in \cite[Thm. 2.9]{NY18} that any $C^*$-Frobenius algebra in a $C^*$-tensor category is left isomorphic to a standard special $C^*$-Frobenius algebra (i.e., a standard $Q$-system). See \cite{NY18} for the meanings of these names.
\end{rem}

\begin{df}
Let $(\mc H_a,\mc B^a,\iota^a)$ and $(\mc H_b,\mc B^b,\iota^b)$ be extensions of $\mc A$. An \textbf{isomorphism} $\varphi:\mc B^a\rightarrow\mc B^b$ (with respect to $\mc A$) is a collection of (normal) isomorphisms of von Neumann algebras 
\begin{align*}
\varphi_{\wtd I}:\mc B^a(\wtd I)\xrightarrow{\simeq}\mc B^b(\wtd I)	
\end{align*}
(for all $\wtd I\in\Jtd$) satisfying the compatibility and the extension property as in Def. \ref{lb7} (namely, $(\mc H_b,\varphi)$ is a $(\mc B^a,\mc A)$-module, or equivalently, $(\mc H_a,\varphi^{-1})$ is a $(\mc B^b,\mc A)$-module).

Such $\varphi$ is called a \textbf{unitary isomorphism} if there is a (necessarily unique) unitary operator $V:\mc H_a\rightarrow\mc H_b$ satisfying
\begin{gather}
V\iota^a=\iota^b,\\
V X=\varphi_{\wtd I}(X)V
\end{gather}
for all $\wtd I\in\Jtd,X\in\mc B^a(\wtd I)$.
\end{df}

\begin{proof}
The uniqueness of $V$ follows from the fact that $VX\iota^a\Omega=\varphi_{\wtd I}(X)\iota^b\Omega$.
\end{proof}

The composition of two isomorphisms is clearly a morphism of extensions. Thus, the extensions of $\mc A$ form a category whose objects are the isomorphisms.

\begin{lm}\label{lb10}
Let $Q^a,Q^b$ be as in Def. \ref{lb8}, and choose $V\in\Hom_{\mc A}(\mc H_a,\mc H_b)$ with bounded inverse satisfying $V\iota^a=\iota^b$. Set
\begin{align}
\mu^V=\mu^b(V\boxtimes\id_b)\quad \in\Hom_{\mc A}(\mc H_a\boxtimes\mc H_b,\mc H_b).
\end{align}
Then $V:Q^a\rightarrow Q^b$ is a left isomorphism if and only if $(\mc H_b,\mu^V)$ is a left $Q^a$-module.
\end{lm}

\begin{proof}
The unit property of $(\mc H_b,\mu^V)$ is automatic. The associativity 
\begin{align*}
\mu^V(\id_a\boxtimes\mu^V)=\mu^V(\mu^a\boxtimes\id_b)	
\end{align*}
of $(\mc H_b,\mu^V)$ means
\begin{align*}
\mu^b(\id_b\boxtimes\mu^b)(V\boxtimes V\boxtimes\id_b)=	\mu^b(V\mu^a\boxtimes\id_b).
\end{align*}
By the associativity of $\mu^b$, this is equivalent to
\begin{align*}
	\mu^b(\mu^b(V\boxtimes V)\boxtimes\id_b)=	\mu^b(V\mu^a\boxtimes\id_b).
\end{align*}
By applying $\id_a\boxtimes\id_a\boxtimes\iota^b$ to this relation and using the unit property of $\mu^b$, we see that the above relation (and hence the associativity of $\mu^V$) is equivalent to $V\mu^a=\mu^b(V\boxtimes V)$.

Set $\ev_a=(\iota^a)^*\mu^a\in\Hom_{\mc A}(\mc H_a\boxtimes\mc H_a,\mc H_0)$. Then the Frobenius relation of $\mu^V$ is equivalent to
\begin{align*}
\mu^V=(\ev_a\boxtimes\id_b)(\id_a\boxtimes(\mu^V)^*).
\end{align*}
Since $\mu^V=\mu^b(V\boxtimes\id_b)=V\mu^a(\id_a\boxtimes V^{-1})$, the above relation means
\begin{align*}
\mu^b(V\boxtimes\id_b)=(V^{-1})^*(\ev_a\boxtimes\id_a)(\id_a\boxtimes(\mu^a)^*)(\id_a\boxtimes V^*).
\end{align*}
By the Frobenius relation of $\mu^a$, we have $(\ev_a\boxtimes\id_a)(\id_a\boxtimes(\mu^a)^*)=\mu^a$. So the above relation is equivalent to
\begin{align*}
\mu^b(V\boxtimes\id_b)=(V^{-1})^*\mu^a(\id_a\boxtimes V^*).
\end{align*}
This is equivalent to the third of \eqref{eq18}.
\end{proof}

In the following theorem, we let $Q^a=(\mc H_a,\mu^a,\iota^a)$, $Q^b=(\mc H_b,\mu^b,\iota^b)$, etc. be $C^*$-Frobenius algebras in $\RepA$.

\begin{thm}\label{lb13}
The following are true.
\begin{enumerate}[label=(\alph*)]
\item If $V:Q^a\rightarrow Q^b$ is a left isomorphism of $C^*$-Frobenius algebras, then there is a unique isomorphism $\varphi^V:\mc B_{Q^a}\rightarrow\mc B_{Q^b}$ (with respect to $\mc A$) satisfying
\begin{align}
VX=\varphi_{\wtd I}^V(X)V	\label{eq19}
\end{align}
for all $\wtd I\in\Jtd,X\in\mc B_{Q^a}(\wtd I)$. Moreover, $V$ is unitary if and only if $\varphi^V$ is so. 
\item If $W:Q^b\rightarrow Q^c$ is also a left isomorphism, then $\varphi^{WV}=\varphi^W\circ\varphi^V$.
\item If $(\mc H_b,\mc B^b,\iota^b)$ is an extension of $\mc A$, and if $\varphi:\mc B_{Q^a}\rightarrow\mc B^b$ is an isomorphism (with respect to $\mc A$), then $\mc B^b=\mc B_{Q^b}$ for a unique $C^*$-Frobenius algebra $Q^b$, and $\varphi=\varphi^V$ for a unique left isomorphism $V:Q^a\rightarrow Q^b$.
\end{enumerate}
\end{thm}

\begin{proof}
(a) The uniqueness of $\varphi^V$ follows from the surjectivity of $V$. As for the existence, by Thm. \ref{lb6} (applied to $\mu^V$) and Lemma \ref{lb10}, we have a $(\mc B_{Q^a},\mc A)$-module $(\mc H_b,\varphi^V)$ such that for all $\wtd I\in\Jtd$ and $\xi\in\mc H_a(I)$,
\begin{align}
\varphi^V_{\wtd I}\big(\mu^aL(\xi,\wtd I)|_{\mc H_a}\big)=\mu^b L(V\xi,\wtd I)|_{\mc H_b}.	
\end{align}
Clearly $\varphi^V_{\wtd I}(\mc B_{Q^a}(\wtd I))=\mc B_{Q^b}(\wtd I)$. Since we have a similar homomorphism $\varphi_{\wtd I}^{V^{-1}}$ sending each $\mu^b L(V\xi,\wtd I)|_{\mc H_b}$ to $\mu^aL(\xi,\wtd I)|_{\mc H_a}$, $\varphi^V_{\wtd I}$ must be an isomorphism of von Neumann algebras. Thus, $\varphi^V$ is an isomorphism from $\mc B_{Q^a}$ to $\mc B_{Q^b}$. \eqref{eq19} holds since
\begin{align*}
V\mu^a L(\xi,\wtd I)|_{\mc H_a}=\mu^b(V\boxtimes V)L(\xi,\wtd I)|_{\mc H_a}=\mu^b L(V\xi,\wtd I)|_{\mc H_b}\cdot V.	
\end{align*}

If $V$ is unitary, then $\varphi^V$ is clearly so. Conversely, if $\varphi^V$ is unitary, then there is a unitary  $W:\mc H_a\rightarrow\mc H_b$ satisfying $W\iota^a=\iota^b$ and $W(\cdot )W^{-1}=\varphi_{\wtd I}$. Therefore, for each $X\in\mc B_{Q^a}(\wtd I)$, both $W$ and $V$ send $X\iota^a\Omega$ to $\varphi_{\wtd I}(X)\iota^b\Omega$. Therefore $V=W$, and hence $V$ is unitary. 

(b) is obvious.

(c) Let $\varphi:\mc B_{Q^a}\rightarrow\mc B^b$ be an isomorphism. Since $(\mc H_b,\varphi)$ is a $(\mc B_{Q^a},\mc A)$-module, by Thm. \ref{lb6}, we have a left $Q^a$-module $(\mc H_b,\wht\mu)$ (where $\wht\mu\in\Hom_{\mc A}(\mc H_a\boxtimes\mc H_b,\mc H_b)$) such that for each $\wtd I\in\Jtd,\xi\in\mc H_a(I)$,
\begin{align*}
\varphi_{\wtd I}\big(\mu^aL(\xi,\wtd I)|_{\mc H_a}\big)=\wht\mu L(\xi,\wtd I)|_{\mc H_b}.	
\end{align*}
In particular, each $Y\in\mc B^b(\wtd I)$ is of the form $Y=\wht\mu L(\xi,\wtd I)|_{\mc H_b}$. Set
\begin{align}
V=\wht\mu(\id_a\boxtimes\iota^b)\quad\in\Hom_{\mc A}(\mc H_a,\mc H_b).	\label{eq21}
\end{align}
Then
\begin{align*}
Y\iota^b\Omega=	\wht\mu L(\xi,\wtd I)\iota^b\Omega=\wht\mu(\id_a\boxtimes\iota^b)L(\xi,\wtd I)\Omega=V\xi.
\end{align*}
Thus, $V$ has dense range since $\mc B^b(\wtd I)\iota^b\Omega$ is dense.

Let us show that $V$ has bounded inverse. Note that $V$ and $V^*V$ have the same null space. Also, since $V^*V\in\End_{\mc A}(\mc H_a)$ where $\End_{\mc A}(\mc H_a)$ is finite-dimensional (as $\mc H_a$ is dualizable), the spectrum of $V^*V$ must be a finite set. Thus, if we can show that $V^*V$ has trivial null space, then $V^*V$ and hence $(V^*V)^{\frac 12}$ must have bounded inverse. Then by polar decomposition, $V$ would have bounded inverse.

Suppose the null space of $V^*V$ is non-trivial. Since it is an $\mc A$-submodule of $\mc H_a$,  it must contain a non-zero element $\xi\in\mc H_a(I)$. Thus, as $V\xi=0$, we have $Y\iota^b\Omega=0$ where $Y=\wht\mu L(\xi,\wtd I)|_{\mc H_b}$. By the Reeh-Schlieder property for $\mc B^b$, we have $Y=0$. So $\varphi_{\wtd I}\big(\mu^aL(\xi,\wtd I)|_{\mc H_a}\big)=0$. As $\varphi_{\wtd I}$ is faithful, $\mu^a L(\xi,\wtd I)|_{\mc H_a}=0$. So $\xi=0$ by Rem. \ref{lb27}. This is a contradiction. Thus, we have finished proving that $V$ has bounded inverse.

Now, for each $\eta\in\mc H_b$ and $Y=\wht\mu L(\xi,\wtd I)|_{\mc H_b}$,
\begin{align*}
L(Y\iota^b\Omega,\wtd I)\eta=L(V\xi,\wtd I)\eta=(V\boxtimes\id_b)L(\xi,\wtd I)\eta,
\end{align*}
and
\begin{align*}
Y\eta=\wht\mu L(\xi,\wtd I)\eta.	
\end{align*}
Thus we have a bounded map, namely $\wht\mu(V^{-1}\boxtimes\id_b)$, from $\mc H_b\boxtimes\mc H_b$ to $\mc H_b$ sending
\begin{align*}
L(Y\iota^b\Omega,\wtd I)\eta\mapsto Y\eta
\end{align*}
for each $\wtd I\in\Jtd$ and $Y\in\mc B^b(\wtd I)$. Thus, as argued in Step 2 of the proof of Thm. \ref{lb11}, $\mu^b:=\wht\mu(V^{-1}\boxtimes\id_a)$ defines a $C^*$-Frobenius algebra structure $Q^b=(\mc H_a,\mu^b,\iota^b)$, and $\mc B^b=\mc B_{Q^b}$. By \eqref{eq21} and the unit property of $\wht\mu$, we have $V\iota^a=\iota^b$. As $\wht\mu=\mu^b(V\boxtimes\id_a)$ equals the $\mu^V$ in Lemma \ref{lb10}, by that lemma, $V$ is a left isomorphism from $Q^a$ to $Q^b$. 

Finally, the uniqueness of isomorphisms $V:Q^a\rightarrow Q^b$ satisfying $\varphi=\varphi^V$ (namely, satisfying $VX=\varphi_{\wtd I}(X)V$ for all $X\in\mc B_{Q^a}(\wtd I)$) is due to
\begin{align}
VX\iota^a\Omega=\varphi_{\wtd I}(X)\iota^b\Omega.
\end{align}
\end{proof}

\begin{rem}
In Def. \ref{lb8}, if we replace the third relation of \eqref{eq18} by
\begin{align}
V^*\mu^b(\id_b\boxtimes V)=\mu^a(V^*\boxtimes\id_a)	
\end{align}
(this is equivalent to $V^*V\mu_a=\mu_a(V^*V\boxtimes\id_a)$), then such $V:Q^a\rightarrow Q^b$ is called a \textbf{right isomorphism}. In Thm. \ref{lb13}, if we replace $\mc B_{Q^a},\mc B_{Q^a}$  by $\mc B'_{Q^a},\mc B'_{Q^b}$   (cf. Rem. \ref{lb14}) and replace the left isomorphisms by the right isomorphisms, then the theorem still holds.
\end{rem}

Similar to Main Thm. \ref{lb12}, we formulate (part of) Thm. \ref{lb13} in terms of isomorphism of categories. Let $\mathrm{Frob}^{\mathrm L}(\mc A)$ be the category of $C^*$-Frobenius algebras in $\RepA$ whose morphisms are the left isomorphisms. Let $\mathrm{Ext}^{\mathrm d}(\mc A)$ be the category of finite index extensions of $\mc A$ whose morphisms are the isomorphisms (with respect to $\mc A$).

\begin{Mthm}\label{lb26}
The functor $\fk G:\mathrm{Frob}^{\mathrm L}(\mc A)\rightarrow\mathrm{Ext}^{\mathrm d}(\mc A)$ sending each $Q$ to $\mc B_Q$ and each left isomorphism $V$ to $\varphi^V$ is an isomorphism of categories.
\end{Mthm}

\begin{rem}\label{lb15}
Let us translate the results of these two sections to the language of subfactors and von Neumann bimodules. We fix a von Neumann factor $\mc N$ together with a faithful normal state $\omega$. We call $\mc H_0:=L^2(\mc N,\omega)$ the vacuum $\mc N-\mc N$ bimodule. The element $1$ in $L^2(\mc N,\omega)$ is denoted by $\Omega$. An \textbf{abstract extension} of $\mc N$ is an inclusion $\mc N\subset\mc M$ (where $\mc M$ is a von Neumann algebra) such that there exists a faithful normal conditional expectation $\mc E:\mc M\rightarrow\mc N$. (We only assume the existence of $\mc E$ but do not include it as part of the data of an abstract extension.) A \textbf{concrete extension} of $\mc N$ is $(\mc H_a,\mc N\subset\mc M,\iota)$, where $\mc M$ is a von Neumann algebra containing $\mc N$, $\mc H_a$ is an $\mc M-\mc N$ bimodule, $\iota:\mc H_0\rightarrow\mc H_a$ is an isometric homomorphism of $\mc N-\mc N$ bimodules, and $\iota\Omega$ is cyclic and separating under the left action of $\mc M$.

A concrete extension $(\mc H_a,\mc N\subset\mc M,\iota)$ determines a faithful normal conditional expectation $\mc E:\mc M\rightarrow\pi_a(\mc N)$ (where $\pi_a$ is the left representation of $\mc N$ on $\mc H_a$) satisfying
\begin{align*}
\iota\iota^* X\iota\iota^*=\mc E(X)\iota\iota^*.	
\end{align*}
We say this extension has finite index if $\mc E$ has finite index (in the sense of Pimsner-Popa condition). Conversely, any abstract extension $\mc N\subset \mc M$ with a chosen faithful normal conditional expectation $\mc E$ determines a concrete extension $(L^2(\mc M,\omega\circ\mc E),\mc N\subset\mc M,\iota)$ where $\iota$ is the canonical embedding $L^2(\mc N,\omega)\rightarrow L^2(\mc M,\omega\circ\mc E)$. Thus, concrete extensions are roughly the same as abstract extensions with chosen conditional expectations. 

A (resp. finite index) concrete extension of $\mc N$ is analogous to a (resp. finite index) extension of conformal net $\mc A$.  On each side, a finite index extensions is described uniquely by a $C^*$-Frobenius algebra $Q$. The analogous result of Main Thm. \ref{lb12} for a  concrete extension $(\mc H_a,\mc N\subset\mc M,\iota)$ determined by $Q$ (in the $C^*$-tensor category of $\mc N-\mc N$ bimodules) is the well known fact that the $C^*$-category of left $Q$-modules is canonically isomorphic to the $C^*$-category of $\mc M-\mc N$ bimodules.

We may define isomorphisms of concrete extensions $(\mc H_a,\mc N\subset\mc M^a,\iota^a),(\mc H_b,\mc N\subset\mc M^b,\iota^b)$ to be the isomorphisms of von Neumann algebras $\mc M^a\rightarrow\mc M^b$ that restrict to the identity map on $\mc N$. Unitary isomorphisms are those implemented by unitary maps $V:\mc H_a\rightarrow\mc H_b$ such that $V\iota^a=\iota^b$. Then the isomorphism classes of concreted extensions are roughly the same as abstract extensions. The analogous result of Thm. \ref{lb13} is the following:
\begin{enumerate}[label=(\alph*)]
	\item If $(\mc H_a,\mc N\subset\mc M^a,\iota^a),(\mc H_b,\mc N\subset\mc M^b,\iota^b)$ are isomorphic, and if the first concrete extension is realized by a $C^*$-Frobenius algebra $Q^a$, then the second one is realized by some $Q^b$.
	\item Left isomorphisms from $Q^a$ to $Q^b$ correspond bijectively to isomorphisms from the first concrete extension to the second one. A left isomorphism is unitary iff the corresponding isomorphism of concrete extensions is unitary.
\end{enumerate}
We do not prove these statements in this article. But they can be proved using  arguments similar to  those in this article.

In the more familiar subfactor language, (a) says the well known fact that if $\mc N\subset\mc M$  has a finite index faithful normal conditional expectation $\mc E_1$, then any other faithful normal  conditional expectation $\mc E_2:\mc M\rightarrow\mc N$ also has finite index (i.e., can be realized by $C^*$-Frobenius algebras). We have used this result in the proof of Thm. \ref{lb11}. But we have avoided using this when proving Thm. \ref{lb13}. (b) says that two extensions $(\mc N\subset\mc M^a,\mc E^a)$ and $(\mc N\subset\mc M^b,\mc E^b)$ with the data of finite index faithful normal conditional expectations, together with an isomorphism of von Neumann algebras $\mc M^a\rightarrow\mc M^b$ fixing $\mc N$, corresponds  to a left isomorphism of $C^*$-Frobenius algebras. If the isomorphism $\mc M^a\rightarrow\mc M^b$ intertwines $\mc E^a$ and $\mc E^b$, then it corresponds to a unitary  isomorphism of $C^*$-Frobenius algebras.

In particular, by choosing $\mc M^a=\mc M^b=\mc M$ and choosing the isomorphism to be $\id$, we see that  an ordered pair of two finite index faithful normal conditional expectations $(\mc E_1,\mc E_2)$ for $\mc N\subset\mc M$ corresponds to a (non-necessarily unitary) left isomorphism of $C^*$-Frobenius algebras. Therefore, the study of the faithful normal conditional expectations for a given abstract finite index extension $\mc N\subset\mc M$ can be transformed  to that of  the left isomorphisms of $C^*$-Frobenius algebras. 

We remark that the relation between  concrete extensions and  extensions with fixed conditional expectations is similar to that between Connes fusion and the theory of (superselection) sectors: the former ``Schr\"odinger picture" focuses  on  representations of von Neumann algebras on Hilbert spaces, and the latter ``Heisenberg picture" focuses on maps (morphisms and conditional expectations) between von Neumann algebras.
\end{rem}

\section{An example}

In this section, we choose a dualizable $\mc A$-module $\mc H_i$ with dual object $\mc H_{\ovl i}$ and $\mc A$-module morphisms $\ev_{\ovl i,i}:\mc H_{\ovl i}\boxtimes\mc H_i\rightarrow\mc H_0,\ev_{i,\ovl i}:\mc H_i\boxtimes\mc H_{\ovl i}\rightarrow\mc H_0$ satisfying the conjugate equations \eqref{eq22}. Then we have a net of Jones-Wassermann subfactors $I\mapsto \pi_{i,I}(\mc A(I))\subset \pi_{i,I'}(\mc A(I'))'$. This is not an extension in our sense (cf. Def. \ref{lb16}). But we shall see that it is a module of a finite index extension $\mc B_Q$ of $\mc A$.

Set $\coev_{i,\ovl i}=(\ev_{i,\ovl i})^*$ and $\coev_{\ovl i,i}=(\ev_{\ovl i,i})^*$. For each $\xi\in\mc H_i(I)$, set
\begin{align*}
S_{\wtd I}\xi=L(\xi,\wtd I)^*\coev_{i,\ovl i}\Omega.
\end{align*}
Then, as $L(\xi,\wtd I)^*\coev_{i,\ovl i}\in\Hom_{\mc A}(\mc H_0,\mc H_{\ovl i})$, we have $S_{\wtd I}\xi\in\mc H_{\ovl i}(I)$. We need a special case of \cite[Cor. 5.6]{Gui21b}:
\begin{align}
\ev_{\ovl i,i}L(S_{\wtd I}\xi,\wtd I)|_{\mc H_i}=L(\xi,\wtd I)^*|_{\mc H_i}.	\label{eq23}
\end{align}
We give a proof here since it is straightforward: we have
\begin{align*}
&\ev_{\ovl i,i}L(S_{\wtd I}\xi,\wtd I)|_{\mc H_i}=\ev_{\ovl i,i}L(L(\xi,\wtd I)^*\coev_{i,\ovl i}\Omega,\wtd I)|_{\mc H_i}\\	
=&\ev_{\ovl i,i}L(\xi,\wtd I)^*(\coev_{i,\ovl i}\boxtimes\id_i)L(\Omega,\wtd I)|_{\mc H_i}=L(\xi,\wtd I)^*(\id_i\boxtimes\ev_{\ovl i,i})(\coev_{i,\ovl i}\boxtimes\id_i)|_{\mc H_i},
\end{align*}
which equals $L(\xi,\wtd I)^*$ by the conjugate equations \eqref{eq22}. 

We define a $C^*$-Frobenius algebra $Q=(\mc H_i\boxtimes\mc H_{\ovl i},\mu,\coev_{i,\ovl i})$ where
\begin{align*}
\mu=\id_i\boxtimes\ev_{\ovl i,i}\boxtimes\id_{\ovl i}\quad\in\Hom_{\mc A}(\mc H_i\boxtimes\mc H_{\ovl i}\boxtimes\mc H_i\boxtimes\mc H_{\ovl i},\mc H_i\boxtimes\mc H_{\ovl i}).	
\end{align*}
Then $(\mc H_i,\mu_i)$ is a left $Q$-module where
\begin{align*}
	\mu_i=\id_i\boxtimes\ev_{\ovl i,i}\quad\in\Hom_{\mc A}(\mc H_i\boxtimes\mc H_{\ovl i}\boxtimes\mc H_i,\mc H_i).	
\end{align*}
Let $(\mc H_i,\pi_i)$ be the corresponding $(\mc B_Q,\mc A)$-module defined by Thm. \ref{lb6}.

\begin{pp}
For each $\wtd I\in\Jtd$, we have
\begin{align*}
\pi_{i,\wtd I}(\mc B_Q(\wtd I))=\pi_{i,I'}(\mc A(I'))'.	
\end{align*}
\end{pp}

Therefore, the subfactors $\pi_{i,I}(\mc A(I))\subset \pi_{i,I'}(\mc A(I'))'$ and $\pi_{i\boxtimes\ovl i,I}(\mc A(I))\subset\mc B_Q(\wtd I)$ are isomorphic.

\begin{proof}
The $\subset$ is obvious due to the relative locality (cf. Rem. \ref{lb17}). Choose any $T\in \pi_{i,I'}(\mc A(I'))'$. We claim that $T=L(\eta,\wtd I)L(\xi,\wtd I)^*|_{\mc H_i}$ for some $\xi,\eta\in\mc H_i(I)$. Indeed, choose any unitary vector $\xi\in\mc H_i(I)$ (cf. Lemma \ref{lb2}). Since $TL(\xi,\wtd I)|_{\mc H_0}:\mc H_0\rightarrow\mc H_i$ intertwines the actions of $\mc A(I')$, $\eta:=TL(\xi,\wtd I)\Omega$ belongs to $\mc H_i(I)$. The relation $L(\eta,\wtd I)|_{\mc H_0}=TL(\xi,\wtd I)|_{\mc H_0}$ holds when acting on $\Omega$. Thus, it holds when acting on $\mc A(I')\Omega$ since both sides intertwine the actions of $\mc A(I')$. This proves $T=L(\eta,\wtd I)L(\xi,\wtd I)^*|_{\mc H_i}$ by the unitarity of $L(\xi,\wtd I)$. 

Now, by \eqref{eq23},
\begin{align*}
&T=L(\eta,\wtd I)\ev_{\ovl i,i}L(S_{\wtd I}\xi,\wtd I)|_{\mc H_i}=(\id_i\boxtimes \ev_{\ovl i,i})L(\eta,\wtd I)L(S_{\wtd I}\xi,\wtd I)|_{\mc H_i}\\
=&\mu_i L(L(\eta,\wtd I)S_{\wtd I}\xi,\wtd I)|_{\mc H_i}.
\end{align*}
By \eqref{eq4}, $T$ equals $\pi_{i,\wtd I}(X)$ where $X=\mu L(L(\eta,\wtd I)S_{\wtd I}\xi,\wtd I)|_{\mc H_i\boxtimes\mc H_{\ovl i}}$ belongs to $\mc B_Q(\wtd I)$ (cf. \eqref{eq28}).
\end{proof}

\section{Commutative $C^*$-Frobenius algebras and local extensions}\label{lb19}

In this section, we fix a $C^*$-Frobenius algebra $Q=(\mc H_a,\mu,\iota)$, and let $(\mc H_a,\mc B_Q)$ be the extension of $\mc A$ associated associated to $Q=(\mc H_a,\mu,\iota)$.

\begin{df}
$Q$ is called \textbf{commutative} if $\mu\ss_{a,a}=\mu$. (Recall that $\ss$ is the braid operator.) It is called \textbf{irreducible} if it is an irreducible left $Q$-module.
\end{df}

Note that by \cite[Rem. 2.7]{NY18}, $Q$ is irreducible if and only if $\dim\Hom_{\mc A}(\mc H_0,\mc H_a)=1$

\begin{pp}\label{lb21}
Fix any $\wtd I_0\in\Jtd$. Then the following are equivalent
\begin{enumerate}[label=(\arabic*)]
\item $\mc B_Q(\wtd I_0)'=\mc B_Q(\wtd I_0')$.
\item $\mc B_Q(\wtd I)'=\mc B_Q(\wtd I')$ for every $\wtd I\in\Jtd$.
\item $Q$ is commutative. 
\end{enumerate}
If any of these holds, we say $\mc B_Q$ is a \textbf{local extension} of $\mc A$. If $\mc B_Q$ is local, then $\mc B_Q(\wtd I)$ depends only on $I$ but not on the arg function $\arg_I$. So we may write $\mc B_Q(\wtd I)$ as $\mc B_Q(I)$.
\end{pp}

\begin{proof}
Note that 	for each $\eta\in\mc H_i(I')$, $\mu\ss_{a,a}L(\eta,\wtd I')|_{\mc H_a}=\mu R(\eta,\wtd I')|_{\mc H_a}$ by the braiding axiom \eqref{eq24}. Also (2)$\Rightarrow$(1) is obvious. 

(3)$\Rightarrow$(2): If $\mu=\mu\ss_{a,a}$, then $\mu L(\eta,\wtd I')|_{\mc H_a}=\mu R(\eta,\wtd I')|_{\mc H_a}$.  By \eqref{eq20}, $\mu R(\eta,\wtd I')|_{\mc H_a}\in\mc B_Q(\wtd I)'$, and any element of $\mc B_Q(\wtd I)'$ is of this form. Thus $\mc B_Q(\wtd I')=\mc B(\wtd I)'$.

(1)$\Rightarrow$(3): Suppose $\mc B_Q(\wtd I_0')=\mc B_Q(\wtd I_0)'$. For each $\eta\in\mc H_i(I_0')$, $Y=\mu L(\eta,\wtd I_0')|_{\mc H_a}$	is an element of $\mc B(\wtd I_0)'$, which by \eqref{eq20} is of the form $\mu R(\eta_0,\wtd I_0')|_{\mc H_a}$ for some $\eta_0\in\mc H_i(I_0')$. We have $\eta_0=\eta$ since $Y\iota\Omega=\mu L(\eta,\wtd I_0')\iota\Omega=\mu(\id_a\boxtimes\iota)L(\eta,\wtd I_0')\Omega=\eta$ and similarly $Y\iota\Omega=\eta_0$. So $\mu L(\eta,\wtd I_0')|_{\mc H_a}=\mu\ss_{a,a}L(\eta,\wtd I_0')|_{\mc H_a}$. By the density of fusion product, we get $\mu=\mu\ss_{a,a}$.

If $\mc B_Q$ is local, then both $\mc B_Q(\wtd I)$ and $\mc B_Q(\wtd I'')=\mc B_Q(\varrho(-2\pi)\wtd I)$ are the commutants of $\mc B_Q(\wtd I')$. So they are equal. This proves that $\mc B_Q(\wtd I)$ is independent of $\arg_I$.
\end{proof}

\begin{df}
A \textbf{dyslectic $Q$-module} $(\mc H_i,\mu_i)$ is by definition a left $Q$-module satisfying $\mu_i\ss_{i,a}=\mu_i\ss_{a,i}^{-1}$. Thus, $\mc H_i$ is automatically a $Q$-bimodule with left action $\mu_i$ and right action $\mu_i\ss_{i,a}$.
\end{df}

\begin{df}
When $\mc B_Q$ is local, a $(\mc B_Q,\mc A)$-module $(\mc H_i,\pi_i)$ is called \textbf{dyslectic} if $\pi_{i,\wtd I}$ depends only on $I$ but not on $\arg_I$. So we may write $\pi_{i,\wtd I}$ as $\pi_{i,I}$. A dyslectic $(\mc B_Q,\mc A)$-modules is simply called a \textbf{$\mc B_Q$-module}, which is in line with the definition of $\mc A$-modules in Sec. \ref{lb18}.
\end{df}

\begin{pp}\label{lb22}
Let $Q$ be commutative. Let $(\mc H_i,\mu_i)$ be a left $Q$-module, and let $(\mc H_i,\pi_i)$ be the corresponding $(\mc B_Q,\mc A)$-module as in Thm. \ref{lb6}. Then $(\mc H_i,\mu_i)$ is dyslectic if and only if $(\mc H_i,\pi_i)$ is so.
\end{pp}

\begin{proof}
For each $\wtd I\in\Jtd$, let $\wtd I_1=\varrho(2\pi)\wtd I$. For each  $X=\mu L(\xi,\wtd I)|_{\mc H_a}=\mu L(\xi_1,\wtd I_1)|_{\mc H_a}$ in $\mc B_Q(I)$,  $\xi$ and $\xi_1$ must be equal since both equal $X\iota\Omega$ by Rem. \ref{lb27}. By Lemma \ref{lb20}, we have
\begin{gather*}
\pi_{i,\wtd I}(X)=\mu_i L(\xi,\wtd I)|_{\mc H_a},\qquad \pi_{i,\wtd I_1}(X)=\mu_i\ss_{i,a}\ss_{a,i}  L(\xi,\wtd I)|_{\mc H_a}.	
\end{gather*}
Thus, by the density of fusion product, we have $\pi_{i,\wtd I}=\pi_{i,\wtd I_1}$ iff $\mu_i=\mu_i\ss_{i,a}\ss_{a,i}$.
\end{proof}

\begin{df}
We say that $\mc B_Q$ is a \textbf{local M\"obius extension} of $\mc A$ if $\mc B_Q$ is a (local) M\"obius net, and if the representation of $\PSU$ on $\mc H_0$ extends to that of $\mc H_a$. Similarly, when $\mc A$ is a conformal net, we say that $\mc B_Q$ is a \textbf{local conformal extension} of $\mc A$ if $\mc B_Q$ is a conformal net, and if the projective representation of $\Diffp(\Sbb^1)$ on $\mc H_0$ extends to that of $\mc H_a$. In either case, we let $\iota\Omega$ be the vacuum vector for $\mc B_Q$.
\end{df}


Since any rigid braided $C^*$-tensor category has a canonical ribbon structure \cite{Mug00}, we have a unitary \textbf{twist operator} $\vartheta_a\in\End_{\mc A}(\mc H_a)$ which commutes with $\End_{\mc A}(\mc H_a)$. (Note that $\mc H_a$ is automatically dualizable.) When $\mc H_a$ is M\"obius covariant, by the conformal spin-statistics theorem (\cite[Thm. 3.13]{GL96}, \cite[Sec. 4.1]{Jor96}, or \cite[Thm. 6.8]{Gui21b}), $\vartheta_a$ equals the action of $\varrho(2\pi)\in\UPSU$ (the rotation by $2\pi$).

\begin{thm}\label{lb25}
Let $Q=(\mc H_a,\mu,\iota)$ be a $C^*$-Frobenius algebra in $\RepA$. Then $\mc B_Q$ is a local M\"obius extension of $\mc A$ with vacuum vector $\iota\Omega$ if and only if the following are true:
\begin{enumerate}[label=(\arabic*)]
\item $\mc H_a$ is a M\"obius covariant $\mc A$-module.
\item $Q$ is commutative and the twist $\vartheta_a=\id_a$.
\end{enumerate}
Moreover, $\mc B_Q$ is irreducible if and only if $Q$ is irreducible. If $\mc A$ is a conformal net, then (1) automatically holds, and (2) is equivalent to that $\mc B_Q$ is a local conformal extension.
\end{thm}

Note that when $Q$ is commutative and irreducible, it is automatic that $\vartheta_a=\id_a$. See \cite[Prop. 2.22]{CGGH23}.

\begin{proof}

$\mc H_a$ is automatically a M\"obius  covariant $\mc A$-module if $\mc B_Q$ is a local M\"obius extension. So we may always assume (1) in the rest of  the proof. Then by \cite[Cor. 4.4]{BCL98} or (in the case that $\mc A$ is a conformal net) \cite{Wei06}, the generator $L_0$ of the rotation subgroup $\varrho$ has positive spectrum when acting on $\mc H_a$. The representation of $\UPSU$ descends to a true representation of $\PSU$ iff $\varrho(2\pi)$ acts as $\id_a$. Therefore, the equivalence of (2) and that $\mc B_Q$ is a local M\"obius resp. conformal extension follows from Prop. \ref{lb21}  and the covariance property of $\mc B_Q$ stated in Thm. \ref{lb5}. The equivalence of the two irreducible conditions follows from Main Thm. \ref{lb12}.
\end{proof}

We now assume $Q$ is commutative and irreducible. Let $\Rep^0(Q)$ be the full $C^*$-subcategory of $\RepL(Q)$ consisting of  dyslectic $Q$-modules. By considering each object of $\Rep^0(Q)$ as $Q-Q$ bimodules, $\Rep^0(Q)$ becomes naturally a braided $C^*$-tensor category: For each dyslectic $Q$-modules $(\mc H_i,\mu_i),(\mc H_j,\mu_j)$, we choose an dyslectic $Q$-module $\mc H_{ij}$ and a surjective left $Q$-module morphisms 
\begin{align}
\mu_{i,j}: \mc H_i\boxtimes\mc H_j\rightarrow \mc H_{ij}\equiv\mc H_i\boxtimes_Q\mc H_j	\label{eq26}
\end{align}
satisfying 
\begin{gather}
\mu_{i,j}(\mu_i\boxtimes\id_j)=\mu_{i,j}(\id_i\boxtimes\mu_j)(\ss_{a,i}\boxtimes\id_j),\label{eq29}\\
\mu_{i,j}^*\mu_{i,j}=(\mu_i\ss_{i,a}\boxtimes\id_j)(\id_i\boxtimes\mu_j^*)\equiv(\id_i\boxtimes\mu_j)((\mu_i\ss_{i,a})^*\boxtimes\id_j).\label{eq30}
\end{gather}
For instance, we may  take $\mc H_i\boxtimes_Q\mc H_j$ to be the range of the right hand side of \eqref{eq30}. The fusion $F\boxtimes_Q G$ of dyslectic $Q$-module morphisms  $F:\mc H_i\rightarrow\mc H_{i'},G:\mc H_j\rightarrow\mc H_{j'}$ is determined by
\begin{align}
\mu_{i,j}(F\boxtimes G)=(F\boxtimes_Q G)\mu_{i,j}.	
\end{align}
The unitors of $\Rep^0(Q)$ are determined by the fact that after identifying the three dyslectic $Q$-modules $\mc H_a\boxtimes_Q \mc H_i,\mc H_i,\mc H_i\boxtimes_Q\mc H_a$ using the unitors, then $\mu_i$ equals $\mu_{a,i}$ (as described in \eqref{eq26}) and $\mu_i\ss_{i,a}$ equals $\mu_{i,a}$. The associativity isomorphisms are determined by the fact that, after suppressing these isomorphisms, for each dyslectic $Q$-modules $\mc H_i,\mc H_j,\mc H_k$, the following diagram commutes adjointly
\begin{equation}
	\begin{tikzcd}
		\quad \mc H_i\boxtimes\mc H_k\boxtimes\mc H_j\quad \arrow[rr,"{\id_i\boxtimes\mu_{k,j}}"] \arrow[d, "{\mu_{i,k}\boxtimes\id_j}"'] &&\quad \mc H_i\boxtimes(\mc H_k\boxtimes_Q\mc H_j)\quad \arrow[d, "{\mu_{i,kj}}"]\\
		(\mc H_i\boxtimes_Q\mc H_k)\boxtimes\mc H_j\arrow[rr,"{\mu_{ik,j}}"] &&\mc H_i\boxtimes_Q\mc H_k\boxtimes_Q\mc H_j
	\end{tikzcd}
\end{equation}
where each arrow denotes the corresponding fusion product morphism as in \eqref{eq26}. See \cite[Sec. 3.2, 3.4]{Gui19}, especially (3.10) and (3.11) for the above diagram. Finally, the braiding $\ss_{i,j}^Q\in\HomL_Q(\mc H_i\boxtimes_Q\mc H_j,\mc H_j\boxtimes_Q\mc H_i)$ is determined by
\begin{align}
\mu_{j,i}\ss_{i,j}=\ss^Q_{i,j}\mu_{i,j}.	
\end{align}

Assume now that $\mc B:=\mc B_Q$ is an irreducible local M\"obius extension of $\mc A$. Let $\Rep(\mc B)$ be the $C^*$-category of $\mc B$-modules (i.e., dyslectic $(\mc B,\mc A)$-modules), which is a full $C^*$-subcategory of $\RepL(\mc B)$. Then one can use Connes fusion to make $\Rep(\mc B)$ a braided $C^*$-tensor category $(\Rep(\mc B),\boxtimes_{\mc B},\ss^{\mc B})$  as in Sec. \ref{lb18} for $\RepA$. In particular, this braided $C^*$-tensor structure is determined by the existence of a categorical extension $\scr E^{\mc B}$ on the $C^*$-category $\Rep(\mc B)$. We let $L^{\mc B},R^{\mc B}$ denote the $L$ and $R$ operators of $\scr E^{\mc B}$. As in the previous sections, we do not write superscripts for the $L$ and $R$ operators of the categorical extension of $\mc A$ over $\RepA$.

Given a $\mc B$-module $\mc H_i$, recall that $\mc H_i(I)=\Hom_{\mc A}(\mc H_i,\mc H_0)\Omega$. Similarly, we define $\mc H_i^{\mc B}(I)=\Hom_{\mc B}(\mc H_i,\mc H_a)\iota\Omega$.
\begin{lm}
We have $\mc H_i(I)=\mc H_i^{\mc B}(I)$.
\end{lm}
\begin{proof}
$\mc H_i(I)\supset\mc H_i^{\mc B}(I)$ since $\Hom_{\mc A}(\mc H_i,\mc H_0)\supset\Hom_{\mc B}(\mc H_i,\mc H_a)\iota$. Now choose any $\eta\in\mc H_i(I)$. We have $\mu_i R(\eta,\wtd I)\iota\Omega=\mu_i(\iota\boxtimes\id_i)R(\eta,\wtd I)\Omega=\eta$. Moreover, let $\wtd K$ be the anticlockwise complement of $\wtd J$. Then $\mu_i R(\eta,\wtd I)|_{\mc H_a}$ intertwines the actions of $\mc B(\wtd K)$ since for each $\xi\in\mc H_a(K)$, 
\begin{align*}
	\mu_i R(\eta,\wtd I)\mu L(\xi,\wtd K)|_{\mc H_a}=\mu_i L(\xi,\wtd K)\mu_i R(\eta,\wtd I)|_{\mc H_a}	
\end{align*}
by the diagram \eqref{eq27}. So $\eta\in \mc H_i^{\mc B}(I)$.
\end{proof}

By Prop. \ref{lb22}, the $*$-functor $\fk F:\RepL(Q)\rightarrow\Rep(\mc B,\mc A)$ in Main Thm. \ref{lb12} reduces to an isomorphism of $C^*$-categories $\fk F:\Rep^0(Q)\rightarrow\Rep(\mc B)$. Moreover:

\begin{Mthm}\label{lb24}
Assume that $\mc B:=\mc B_Q$ is an irreducible local M\"obius extension of $\mc A$. Then the $*$-functor $\fk F$ can be extended to a braided $*$-functor $(\fk F,\Phi):\Rep^0(Q)\rightarrow\Rep(\mc B)$ implementing an isomorphism of braided $C^*$-tensor categories. More precisely: We have an operation $\Phi$ associating to any dyslectic $Q$-modules (equivalently, $\mc B$-modules) $\mc H_i,\mc H_j$ a unitary $\mc A$-module morphism
\begin{align}
\Phi_{i,j}:	\mc H_i\boxtimes_{\mc B}\mc H_j\rightarrow\mc H_i\boxtimes_Q\mc H_j
\end{align}
satisfying that for any $\wtd I\in\Jtd,\xi\in\mc H_i(I),\eta\in\mc H_j$,
\begin{gather}
\begin{gathered}
\Phi_{i,j}L^{\mc B}(\xi,\wtd I)\eta=\mu_{i,j}L(\xi,\wtd I)\eta,\\
\Phi_{j,i}R^{\mc B}(\xi,\wtd I)\eta=\mu_{j,i}R(\xi,\wtd I)\eta.
\end{gathered}	
\end{gather}
$\Phi$ is natural, namely, for morphisms of dyslectic $Q$-modules (equivalently, $\mc B$-modules) $F:\mc H_i\rightarrow\mc H_{i'},G:\mc H_j\rightarrow\mc H_{j'}$, we have $\Phi_{i,j}(F\boxtimes_{\mc B}G)=(F\boxtimes_QG)\Phi_{i,j}$. Then  the unitary equivalence of the associators, unitors, and braid operators of the two categories are implemented by $\Phi$.\footnote{The precise statement can be found in \cite{Gui21a} Thm. 3.10-(a,b,c); the $\boxtimes$ and $\boxdot$ in that theorem correspond respectively to the $\boxtimes_{\mc B}$ and $\boxtimes_Q$ here.}
\end{Mthm}

\begin{proof}
We identify $\Rep^0(Q)$ with $\Rep(\mc B)$ via $\fk F$ so that they can be viewed as the same $C^*$-category. By \cite[Thm. 3.10]{Gui21a}, it suffices to show the existence of a categorical extension $(\mc B,\Rep^0(Q),\boxtimes_Q,\mc H)$ over $\Rep^0(Q)$. For each $\mc B$-module $\mc H_i$, $\wtd I\in\Jtd$, and $\xi\in\mc H_i(I)$, we define the $L$ and $R$ operators acting on each dyslectic $\mc H_k$ to be
\begin{gather*}
L^Q(\xi,\wtd I)=\mu_{i,k} L(\xi,\wtd I):\mc H_k\rightarrow\mc H_i\boxtimes_Q\mc H_k,\\
R^Q(\xi,\wtd I)=\mu_{k,i}R(\xi,\wtd I):\mc H_k\rightarrow\mc H_k\boxtimes_Q\mc H_i.
\end{gather*}
We need to check that they satisfy the axioms in the definition of categorical extensions (cf. Def. \ref{lb23}). The locality axiom follows from the adjoint commutativity of the diagram
\begin{equation}
	\begin{tikzcd}
\mc H_k\arrow[rr,"{R(\eta,\wtd J)}"]\arrow[d,"{L(\xi,\wtd I)}"'] &&\mc H_k\boxtimes\mc H_j\arrow[rr,"\mu_{k,j}"]\arrow[d,"{L(\xi,\wtd I)}"']&&\mc H_k\boxtimes_Q\mc H_j\arrow[d,"{L(\xi,\wtd I)}"']\\
		\mc H_i\boxtimes\mc H_k\arrow[rr,"{R(\eta,\wtd J)}"]\arrow[d,"\mu_{i,k}"']&&\mc H_i\boxtimes\mc H_k\boxtimes\mc H_j\arrow[rr,"\id_i\boxtimes\mu_{k,j}"]\arrow[d,"\mu_{i,k}\boxtimes\id_j"']&&\mc H_i\boxtimes(\mc H_k\boxtimes_Q\mc H_j)\arrow[d,"\mu_{i,kj}"']\\
		\mc H_i\boxtimes_Q\mc H_k\arrow[rr,"{R(\eta,\wtd J)}"]&&\mc(	\mc H_i\boxtimes_Q\mc H_k)\boxtimes\mc H_j\arrow[rr,"\mu_{ik,j}"]&&\mc H_i\boxtimes_Q\mc H_k\boxtimes_Q\mc H_j 
	\end{tikzcd}
\end{equation}
in which $\xi\in\mc H_i(I),\eta\in\mc H_j(J)$ and $\wtd J$ is clockwise to $\wtd I$. In particular, by letting $\mc H_i$ or $\mc H_j$ be $\mc H_a$ and noticing that $\mu_{a,i}=\mu_i,\mu_{i,a}=\mu_i\ss_{i,a}$, we see that both $L^Q(\xi,\wtd I)$ and $R^Q(\xi,\wtd I)$ intertwine the actions of $\mc B(I')$. We now check the other axioms. Isotony is obvious.

Functoriality: For any $G\in\Hom_{\mc B}(\mc H_j,\mc H_{j'})$ and $\xi\in\mc H_i(I),\eta\in\mc H_j$, $(\id_i\boxtimes_Q G)\mu_{i,j}L(\xi,\wtd I)\eta=\mu_{i,j}(\id_i\boxtimes G)L(\xi,\wtd I)\eta=\mu_{i,j}L(\xi,\wtd I)G\eta$. A similar relation holds for the $R^Q$ operators.

State-field correspondence: $\mu_{i,a}L(\xi,\wtd I)\iota\Omega=\mu_i\ss_{i,a}L(\xi,\wtd I)\iota\Omega=\mu_i R(\xi,\wtd I)\iota\Omega=\mu_i(\iota\boxtimes\id_i)R(\xi,\wtd I)\Omega=\xi$, and similarly $\mu_{a,i}R(\xi,\wtd I)\iota \Omega=\xi$.

Density of fusion products: Because $\mu_{i,j}$ is surjective.

Braiding: $\ss^Q_{i,j}\mu_{i,j}L(\xi,\wtd I)\eta=\mu_{j,i}\ss_{i,j}L(\xi,\wtd I)\eta=\mu_{j,i}R(\xi,\wtd I)\eta$.
\end{proof}

\begin{rem}
In Main Thm. \ref{lb24}, it is clear that a dyslectic $Q$-module $\mc H_i$ is dualizable in $\Rep(\mc B)$ iff it is dualizable in $\Rep^0(Q)$. Using induced representations, it is easy to see that the latter is equivalent to that $\mc H_i$ is dualizable in $\RepA$. (See for instance \cite{KO02}, \cite{NY16}, or \cite[Thm. 3.18]{Gui19}.) Therefore, the braided $*$-functor $(\fk F,\Phi)$ in Main Thm. \ref{lb24} restricts to an isomorphism of braided $C^*$-tensor categories
\begin{align*}
(\fk F,\Phi):\Rep^{0,\mathrm d}(Q)\xrightarrow{\simeq}\Rep^{\mathrm d}(\mc B)
\end{align*}
where $\Rep^{0,\mathrm d}(Q)$ is category of all dyslectic $Q$-modules $\mc H_i$ that are dualizable as $\mc A$-modules, and  $\Rep^{\mathrm d}(\mc B)$ is the  category of all  dualizable $\mc B$-modules. Moreover, from \eqref{eq31} and the construction of $(\mc B,\mc A)$-modules from left $Q$-modules,  we see that $\mc H_i\in\Obj(\Rep^0(Q))$ is M\"obius covariant as an $\mc A$-module if and only if it is so as a $\mc B$-module. 
\end{rem}

	
	

	


\noindent {\small \sc Yau Mathematical Sciences Center, Tsinghua University, Beijing, China.}

\noindent {\textit{E-mail}}: binguimath@gmail.com\qquad bingui@tsinghua.edu.cn
\end{document}